\documentclass[11pt,reqno]{amsart}

\usepackage{amssymb,amsmath,amscd,amsthm,xspace,color,tocvsec2, array, caption}
\usepackage[mathscr]{euscript}
\usepackage{enumerate}
\usepackage[breaklinks=true]{hyperref}
\usepackage[all, cmtip]{xy}
\usepackage[ampersand]{easylist}
\usepackage{stmaryrd}
\usepackage[myheadings]{fullpage}
\usepackage{csquotes}
\usepackage{mathtools}

\newcommand{\cgk}{\widehat{\check{\mathfrak{g}}}_{\check{\kappa}}}

\newcommand{\msf}{\mathsf}

\newcommand{\sH}{\mathscr{H}}

\newcommand{\Res}{\operatorname{Res}}

\newcommand{\aff}{\operatorname{aff}}

\newcommand{\Hom}{\operatorname{Hom}}

\newcommand{\g}{\widehat{\mathfrak{g}}}

\newcommand{\fL}{\mathfrak{L}}

\newcommand{\Gr}{\operatorname{Gr}}

\newcommand{\halpha}{\check{\alpha}}

\newcommand{\ft}{\mathfrak{t}}

\newcommand{\id}{\text{id}}

\newcommand{\fg}{\mathfrak{g}}

\newcommand{\comment}[1]{}

\newcommand{\OO}{\mathscr{O}}

\newcommand{\Sym}{\operatorname{Sym}}

\newcommand{\fb}{\mathfrak{b}}
\newcommand{\gk}{\g_\kappa}
\newcommand{\fn}{\mathfrak{n}}

\newcommand{\on}{\operatorname}

\newcommand{\scc}{\mathscr{C}}

\newtheorem{theo}[subsubsection]{Theorem}
\newtheorem*{theo*}{Theorem}
\newtheorem{lem}[subsubsection]{Lemma}
\newtheorem{lemma}[subsubsection]{Lemma}
\newtheorem{cor}[subsubsection]{Corollary}
\newtheorem{pro}[subsubsection]{Proposition}
\newtheorem{conj}[subsubsection]{Conjecture}

\theoremstyle{remark}

\newtheorem{re}[subsubsection]{Remark}

\newtheorem{example}[subsubsection]{Example}

\numberwithin{equation}{section}

\newcommand{\gmk}{\widehat{\fg}_{-\kappa}}
\newcommand{\ck}{\check{\kappa}}

\newcommand{\cg}{\check{G}}
\newcommand{\sJ}{\mathscr{J}}
\newcommand{\cG}{\check{G}}
\newcommand{\cC}{\mathscr{C}}
\newcommand{\fh}{\mathfrak{h}}

\newcommand{\sW}{\EuScript{W}}

\newcommand{\Vect}{\mathsf{Vect}}

\renewcommand{\mod}{\text{\textendash}\mathsf{mod}}
\newcommand{\Whit}{\mathsf{Whit}}
\newcommand{\QCoh}{\mathsf{QCoh}}
\newcommand{\Rep}{\mathsf{Rep}}
\newcommand{\xar}[1]{\xrightarrow{#1}}
\newcommand{\DGCat}{\mathsf{DGCat}}

\begin{document}

\frenchspacing

\title{Fundamental local equivalences in quantum geometric Langlands}

\author{Justin Campbell, Gurbir Dhillon, and Sam Raskin}

\date{\today}

\begin{abstract}

In quantum geometric Langlands, the Satake equivalence plays
a less prominent role than in the classical theory.
Gaitsgory--Lurie proposed a conjectural substitute, later 
termed the \emph{fundamental local equivalence}. 
With a few exceptions, we prove this 
conjecture and its extension to the affine flag variety by using what amount to 
Soergel module techniques.

\end{abstract}

\maketitle

\setcounter{tocdepth}{1}
\tableofcontents

\section{Introduction}

\subsection{}

In the early days of the geometric Langlands program, experts 
observed that the fundamental objects of study deform over the
space of \emph{levels} $\kappa$ for the reductive group $G$. For example,
if $G$ is simple, this is a 1-dimensional space. 
Moreover, levels admit duality as well: a level $\kappa$ for $G$ gives rise
to a \emph{dual} level\footnote{We include a translation by critical levels
in the definition of $\check{\kappa}$, cf. Section \ref{ss:levels} for our 
conventions on levels.} 
$\check{\kappa}$ for the Langlands dual group $\check{G}$.
This observation suggested the existence of a \emph{quantum} geometric Langlands program,
deforming the usual Langlands program.

The first triumph of this idea appeared in the work of Feigin--Frenkel 
\cite{ff-duality}, where they proved duality of affine $\sW$-algebras:
$\sW_{G,\kappa} \simeq \sW_{\check{G},\check{\kappa}}$. We emphasize that
this result is quantum in its nature: the level $\kappa$ appears.
For the \emph{critical} level $\kappa = \kappa_c$, 
which corresponds to classical geometric Langlands, 
Beilinson--Drinfeld \cite{bdh} used Feigin--Frenkel duality to give
a beautiful construction of Hecke eigensheaves for certain irreducible local 
systems. 

\subsection{} 

The major deficiency of the quantum geometric Langlands was 
understood immediately:
the Satake equivalence is more degenerate. 

For instance, in the classical case where 
$\kappa = \kappa_c$ is critical, 
the compatibility of global geometric Langlands
with geometric Satake essentially characterizes the equivalence.
Concretely, this means that for an 
irreducible $\check{G}$-local system $\sigma$
on a smooth, projective curve, one expects there to exist
a canonical Hecke eigensheaf $\EuScript{A}_{\sigma}$ with
eigenvalue $\sigma$.

This does not hold in the quantum setting. For instance,
if $\kappa$ is \emph{irrational},
then the Satake category is equivalent to 
$\Vect$ and the Hecke eigensheaf condition is vacuous. 
For \emph{rational} $\kappa$, 
one hopes for a neutral
gerbe of irreducible Hecke eigensheaves. This is known
for a torus, but the gerbe is 
not at all canonically trivial, cf. 
\cite{lysenko-torus} Section 5.2. For general reductive 
$G$, we are not aware of any conjecture explicitly 
describing the relevant gerbe.

We remark that quantum geometric Langlands for rational $\kappa$ is 
closely tied to the theory of 
automorphic forms on metaplectic groups, 
cf. \cite{gaitsgory-lysenko}. 
It is well-known that 
Hecke eigenvalues form too coarse a decomposition
in the metaplectic setting. Indeed, already 
in his first announcement \cite{hecke-announcement}
of his eponymous operators,
Hecke himself made this observation.

\begin{displayquote}
\vspace{.1cm}
{\it

Es sei zum Schlau\ss{} noch erw\"ahnt, da\ss{} f\"ur die Formen halbzahlinger Dimension, wie 
die einfachen Thetareihen und deren Potenzprodukte, sich eine \"ahnliche Theorie 
nicht aufbauen l\"a\ss{t}.
Da n\"amlich f\"ur diese die Zuordnung zu einer Stufe und zu einer Kongruenzgruppe 
bekanntlich nicht mehr so einfach wie bei ganzzahliger Dimension ist, so kann 
man die Operatoren $T_m$ nur f\"ur Quadratzahlen 
$m=n^2$ definieren, und man erh\"alt so 
einen Zusammenhang nur zwischen den Koeffizienten $a(N)$ und $a(Nn^2)$.

(In conclusion, it should be mentioned that a similar theory cannot be 
developed for forms of half-integral weight, 
such as simple theta functions and monomials in them. 
Since the assignment of a level and a congruence group is not
as easy as for integer dimensions, 
the operators $T_m$ can only be defined 
for square numbers $m = n ^ 2$, 
and a relationship is obtained only between the 
coefficients $a (N)$ and $a (Nn^2)$.)}

\vspace{.1cm}

\end{displayquote}

In the study of metaplectic automorphic forms, one often 
repeatedly finds the guiding principle: 
\emph{it is more fruitful to study Whittaker coefficients than
Hecke eigenvalues}. Indeed, above Hecke exactly observes
the gap that appears between the two in the metaplectic context,
and that Whittaker coefficients provide the finer 
information. See \cite{gps} for a classical application
of these ideas.

We refer to \cite{ggw} for further discussion 
and more recent perspectives. 

\subsection{}

In the geometric setting, Gaitsgory 
made significant advances in applying
the above perspective.  
In \cite{gtw} and \cite{quantum-langlands-summary},
Gaitsgory, pursuing unpublished ideas he developed jointly with
Lurie, formulated a series of conjectures regarding
$\kappa$-twisted Whittaker $D$-modules on the 
affine Grassmannian $\Gr_G = \fL G/\fL^+ G$ and the affine flag 
variety $\fL G/I$.

 Let $\kappa$ be a {nondegenerate} level for $G$ 
 and let 
 $\check{\kappa}$
denote the dual level for $\check{G}$.
Let $\Whit_{\kappa}^{\on{sph}}$ denote the DG category of 
$\kappa$-twisted Whittaker $D$-modules on $\Gr_G$,
and let $\Whit_{\kappa}^{\aff}$ denote the similar
category for $\fL G/I$. 

{ We write $\cgk$ for the affine Lie algebra associated to $\cg$ and 
$\ck$, and  
let
\[ \cgk\mod^{\fL^+ \check{G}} \quad \text{and} \quad 
\cgk\mod^{\check{I}}.
\]
denote the DG categories of Harish--Chandra modules 
for the pairs $(\cgk, \fL^+ \check{G})$ and $(\cgk, \check{I})$, respectively. 
We refer to Section \ref{s:notation} for further clarification
on the notation.}
\begin{conj}[Gaitsgory--Lurie, 
\cite{gtw} Conjecture 0.10]\label{conj:fle-sph}

There is an equivalence of DG categories
\begin{equation} \label{e:gl} 
\Whit_{\kappa}^{\on{sph}}
\simeq 
\cgk\mod^{\fL^+ \cG}.
\end{equation}

\end{conj}

\begin{conj}[Gaitsgory, \cite{quantum-langlands-summary} Conjecture 3.11]\label{conj:fle-tame}

There is an equivalence of DG categories
\begin{equation} \label{e:tfle}
\Whit_{\kappa}^{\aff} 
\simeq 
\cgk\mod^{\check{I}}.
\end{equation}

\end{conj}

We now state a preliminary version of our main result, and then discuss the 
hypothesis that appears in it. 

\begin{theo*}\label{t:main}

Suppose that $\kappa$ is 
\emph{good} in the sense of Section \ref{sss:good-defin}.
Then Conjectures \ref{conj:fle-sph} and \ref{conj:fle-tame}
are true.

\end{theo*}

\subsubsection{} Briefly, a level $\kappa$ is good if and only if 
after restriction 
to every simple factor of $\fg$, $\kappa$ is 
either irrational, or a rational level 
whose denominator is coprime to the bad primes of the root system. We recall 
the latter always lie in $\{ 2, 3, 5\}$. For 
$G$ of type $A$, there are no bad primes, so every level is good in this
case. For an explicit description for general $G$, see Figure \ref{f:badprimes}.

%
%

\subsection{Related works} As emphasized by Gaitsgory in the initial papers 
\cite{gtw} and  
\cite{quantum-langlands-summary}, 
Conjectures \ref{conj:fle-sph} and 
\ref{conj:fle-tame} provide quantum analogues of theorems in classical local 
geometric 
Langlands, i.e. for $\kappa$ or $\check{\kappa}$ at the critical level. We 
presently review these statements.  

\subsubsection{}

For $\kappa = \kappa_c$ critical, Conjecture \ref{conj:fle-sph} 
is known from work of Frenkel--Gaitsgory--Vilonen, and
is a variant of the geometric Satake equivalence.

In this case, heuristically we have
$\check{\kappa} = \infty$;
to interpret this carefully, we refer to
\cite{yifei} for details. Standard arguments then give:

\[
\widehat{\check{\fg}}_{\infty}\mod^{\fL^+ \cG} =
\QCoh(\check{\fg}[[t]]dt/\fL^+ \cG) = \QCoh(\mathbb{B} \cG) =
\Rep(\cG).
\]

\noindent Here the action of $\fL^+ \cG$ on 
$\check{\fg}[[t]]dt$ is the gauge action. 

Then by \cite{fgv}, the composition:

\begin{equation}\label{eq:cs}
\Rep(\cG) \to D(\Gr_G)^{\fL^+ G} \to \Whit(\Gr_G)
\end{equation}

\noindent is an equivalence, where the first functor
is the geometric Satake functor \cite{mirkovic-vilonen}
and the second functor is given by convolution
on the unit object of the right hand side.

We emphasize that unlike with the Satake equivalence,
the equivalence $\Rep(\cG) \xar{\sim} \Whit(\Gr_G)$ is
an equivalence of derived categories, not merely
abelian categories. This amounts to the \emph{cleanness}
property of spherical Whittaker sheaves from \cite{fgv}
and the geometric Casselman--Shalika formula from
\emph{loc. cit}.

\begin{re}
That \eqref{eq:cs} is an equivalence is special to 
integral levels. 
That is, at non-integral levels, the spherical Hecke
category produces only a small part of the spherical Whittaker 
category. This failure, especially at rational levels, is 
part of our interest in Theorem \ref{t:main}.
\end{re}

\subsubsection{}
	
	For $\kappa = \kappa_c$ critical, a version of Conjecture 
	\ref{conj:fle-tame} 
	is the main result of \cite{arkhipov-bezrukavnikov}.
	This deep work of Arkhipov--Bezrukavnikov was
	one of the most significant breakthroughs in geometric
	Langlands, and underlies seemingly countless
	advances in the area since. 

\subsubsection{}
By the above discussion, for $\kappa = \kappa_c$, Conjecture \ref{conj:fle-sph}
is a sort of variant of the geometric Satake theorem. One of 
Gaitsgory's key insights is that in quantum geometric Langlands,
Hecke operators play a diminished role, while Conjecture \ref{conj:fle-sph} 
plays the fundamental role that Satake plays in the classical theory.

\subsubsection{}

For $\check{\kappa} = \check{\kappa}_c$ critical, 
Conjecture \ref{conj:fle-sph} 
is the main result of \cite{fg-sph}, where Frenkel--Gaitsgory construct an 
equivalence

\[
\widehat{\check{\fg}}_{crit}\mod^{\fL^+ \check{G}} \simeq 
\QCoh(\on{Op}_G^{\on{unr}}).
\]

\noindent Here the right hand side is the indscheme of
\emph{unramified (or monodromy-free) opers} for 
$G$; this category is the
$\kappa \to \infty$ limit of $\Whit_{\kappa}(\Gr_G)$ as 
above, and is defined in \emph{loc. cit}. For the centrality
of this result in the geometric Langlands program,
see \cite{dennis-laumonconf} Section 11.

Conjecture \ref{conj:fle-tame} in this case is a folklore
extension of the main results of \cite{dmod-aff-flag} and
\cite{fg-fusion}, but whose complete proof 
is not recorded in the literature.

\subsubsection{}

Finally, let us discuss the previously known cases of Conjectures 
\ref{conj:fle-sph} and \ref{conj:fle-tame}. In the original paper \cite{gtw}, 
Gaitsgory proved  Conjecture \ref{conj:fle-sph} for $\kappa$ irrational, and in 
fact a 
stronger version of it, as we presently describe below. As far as we are aware, 
no other cases of the conjectures have been obtained.  

\subsection{Factorization}

In fact, Gaitsgory conjectured more, related to the
factorization of the Beilinson--Drinfeld affine
Grassmannian.

In Conjecture \ref{conj:fle-sph}, he conjectured
an equivalence of \emph{factorization categories}, 
cf. \cite{gtw} and \cite{chiralcats}. 
Similarly, in Conjecture \ref{conj:fle-tame}, it is expected that
the equivalence should be one of factorization
modules for the (conjecturally equivalent) 
factorization categories appearing in 
Conjecture \ref{conj:fle-sph}. 

When $\kappa = \kappa_c$, these goals are 
implicit in the original work. 
In the spherical case, this is spelled out 
in \cite{cpsi} Theorem 6.36.1.
In the Iwahori case, a weak version of the 
compatibility with factorization module structures
was shown in \cite{cpsii} Theorem 10.8.1.

\subsection{The role of this paper}

In the decade since their formulation, Gaitsgory has been 
advancing an ambitious program
to establish the fundamental local equivalences.
We refer to \cite{paris-notes} for an overview
of this project;
\cite{gtw}, \cite{fact9}, and \cite{fact9} for early work on it;
and \cite{fact3}, \cite{fact6},
\cite{fact7}, \cite{fact5}, \cite{fact1}, \cite{fact2}, and \cite{fact8}
for some of his recent advances in this project.

Gaitsgory's program, though still incomplete, represents 
a new paradigm for Kac-Moody algebras, quantum groups, and 
quantum geometric Langlands. It is full of lovely,
innovative constructions and numerous breakthroughs. 
It is also quite sophisticated, as seems always to be the case
when working with factorization algebras.

Our work is not intended to 
supersede the eventual conclusion of Gaitsgory's project.
Rather, we regard the equivalences of Conjectures \ref{conj:fle-sph} and
\ref{conj:fle-tame} (i.e., forgetting factorization) 
to be interesting results in geometric representation theory
and geometric Langlands. 

For example, as discussed above, the $\kappa = \kappa_c$ analogues of our
results include the geometric Casselman--Shalika formula \cite{fgv}
and the deep work of Arkhipov--Bezrukavnikov \cite{ab}. In fact, as 
we hope to 
explain elsewhere, the geometric part of our study of 
$\Whit_{\kappa}^{\on{sph}}$, suitably adapted to the 
function-field setting, should imply some new function-theoretic results on  
the 
metaplectic Casselman--Shalika formula. 

Moreover, while our present results are expected to be interesting outcomes
of Gaitsgory's methods, we find it desirable to have a 
more direct argument.

\subsection{Methods}\label{ss:methods}

Our techniques are remarkably elementary in comparison to 
the above work of Gaitsgory or e.g. \cite{ab}.
Our main input is classical methods developed by Soergel and his
school.




\subsubsection{}

In his initial work \cite{soergel}, Soergel showed that a
block of Category $\EuScript{O}$ for $\fg$ can be reconstructed 
from the Weyl group of $G$.
Fiebig \cite{fiebig} extended this work to Kac--Moody algebras.
As a consequence of Fiebig's work,
the category $\cgk\mod^{\check{I}}$ can be completely recovered
from the combinatorial datum of the root datum of $\cG$ and the level $\ck$.

To prove Conjecture \ref{conj:fle-tame}, we 
provide a similar Coxeter-theoretic 
description of $\Whit_{\kappa}^{\aff}$. 
We do this by relating $\Whit_{\kappa}^{\aff}$ to 
$\widehat{\fg}_{\kappa}\mod^I$,\footnote{In finite type, an analogous
	result appears in Milicic--Soergel \cite{ms}. 
	Their techniques are not available in the affine setting, so our
	methods differ. 
	
	We use the perspective of 
	loop group actions on categories to
	study Kac--Moody representations. We convolve by an explicit object,
	constructed from $\sW$-algebras.
	
	In contrast, in the finite-type setting,
	\cite{ms} relies on good properties
	of Harish--Chandra bimodules with generalized
	central characters. The theory of affine Harish--Chandra 
	bimodules is in its infancy and is
	much more difficult than in finite type.
	As we hope to explain 
	elsewhere, our methods are sufficient to {\em establish} 
	similar properties of a suitable 
	category of Harish--Chandra bimodules in 
	affine type. However, it should also be possible to prove this equivalence 
	directly by a 
Soergel module argument, cf. Remark \ref{r:Soergelmethods}.} which allows us to 
apply Fiebig's results directly to 
$\Whit_{\kappa}^{\aff}$. 

We then prove Conjecture \ref{conj:fle-tame} 
by matching Langlands dual combinatorics. 
Here we draw the reader's attention to 
Theorem \ref{t:intw}, which is
a combinatorial shadow of quantum Langlands duality.

It is striking that these fundamental conjectures of Gaitsgory
have been open for over a decade, but admit a solution that almost
could have been given at the time. 
\begin{re}

In fact, Theorem \ref{t:intw}, combined 
with the description of twisted Hecke categories as Soergel bimodules, obtained 
in finite type recently in \cite{ly}, 
should yield quantum Langlands duality
for affine Hecke categories. 
Therefore, Soergel's methods, as 
applied in our paper, should suffice to 
prove the local quantum geometric 
Langlands correspondence for categorical representations generated by Iwahori 
invariant vectors. 

\end{re}

\begin{re}

Because of our reliance on \cite{fiebig}, our 
construction is a little non-canonical. 
Indeed, in \emph{loc. cit}., 
there is a choice of projective cover of simple
objects. With that said, 
hewing closer to the Koszul dual picture
as in \cite{ly} would provide canonical equivalences.

\end{re}

\begin{re} After completing this paper, we learned of the thesis of Chris Dodd 
\cite{dodd}, which reproves the results of Arkhipov--Bezrukavnikov \cite{ab} by 
a Soergel 
module argument. Our 
argument may be thought of as a quantum deformation of his approach. We thank 
Roman 
Bezrukavnikov for bringing this to our attention. 
\end{re}

\subsection{Comparison}

In short, we relate Langlands dual categories using Fiebig's
combinatorial description of blocks of affine Category $\EuScript{O}$.

Gaitsgory's program compares these categories via a
factorization algebra $\Omega_q$ (and some of its cousins),
which may also be constructed directly from the root datum of
$G$, cf. \cite{fact4}, \cite{paris-notes} and \cite{fact1}.

It would be quite interesting to find a direct relationship between
these two perspectives.

\begin{re}

We highlight one point of departure in our perspective as compared
to Gaitsgory's. At negative levels, our equivalence is $t$-exact
by construction and matches highest weight structures. 
This was previously anticipated in the spherical case by Gaitsgory,
but was ambiguous in the affine case. After we told him about
our results, he found an argument showing 
that a similar property must hold 
for the equivalence he is working on.

In our approach, these properties are key in deducing the
parahoric version of the theorem from the Iwahori version. 

\end{re}

\vspace{2mm}

\noindent {\bf Acknowledgments.} We would like to thank 
D. Ben-Zvi, R. Bezrukavnikov, A. 
Braverman, D. Bump, D. 
Gaitsgory, S. Kumar, S. Lysenko, I. Mirkovic, W. Wang, B. Webster, D. Yang, and 
Z. Yun for 
interest, encouragement 
and helpful 
discussions. 

Part of this work was carried out at MSRI, where
S.R. was in residence. 
In addition, this research was supported in part by 
Perimeter Institute for Theoretical
Physics. Research at Perimeter Institute is supported by the Government of 
Canada through
the Department of Innovation, Science, and Economic Development, and by the 
Province of
Ontario through the Ministry of Research and Innovation.
We thank all these institutions for their 
hospitality. 

\section{Preliminary material} \label{s:notation}

In this section, we collect standard definitions and notation. We invite the reader 
to skip to the next section and refer back as needed. 

\subsection{Notation for groups}Let $G$ be a reductive group over 
$\mathbb{C}$.\footnote{Our arguments
	apply more generally for any split reductive group over a field $k$
	of characteristic $0$. That is, the cohomology that appears is purely de 
	Rham,
	never \'etale or Betti.} 

\subsubsection{}We fix once and for all a pinning $(T,B,\psi)$ of $G$. 
That is, we fix $T \subset B \subset G$
with $T$ (resp. $B$) a Cartan (resp. Borel) subgroup of $G$.
In addition, for $N$ the unipotent radical of $B$,
we fix a nondegenerate character $\psi:N \to \mathbb{G}_a$.

\subsubsection{}

Given the above data, there is a canonically defined 
Langlands dual group $\check{G}$, which
also comes with a pinning. In particular,
we have Cartan and Borel subgroups
$\check{T} \subset \check{B} \subset \check{G}$.
Again, $\check{N}$ denotes the unipotent radical
of $\check{B}$. 

\subsubsection{}

The data $T \subset B \subset G$ determines a Borel
$B^-$ opposite to $B$, so $B^- \cap B = T$. We denote its
radical by $N^-$. The same applies for $\check{G}$.

\subsubsection{}

We denote the appropriate Lie algebras by 
$\fg, \fb, \fn, \ft, \fb^-, \fn^-, \check{\fg}, \check{\fb}, 
\check{\fn}, \check{\ft}, \check{\fb}^-$, and $\check{\fn}^-$.

\subsubsection{}

We write $\check{\Lambda}_G$ for the lattice
of coweights of $G$, i.e., the cocharacter lattice 
of $T$, and $\Lambda_G$ for the lattice of weights
of $G$, i.e., the character lattice of $T$. We denote the root lattice,
i.e., the integral span of the roots by \[Q \subset \Lambda_G.\]
In other words, $Q = \Lambda_{G^{ad}}$ for $G^{ad}$ the
adjoint group of $G$. Similarly, we let  
$\check{Q} \subset \check{\Lambda}_{\check{G}}$ denote the coroot lattice, and 
one has  $\check{Q} = 
\Lambda_{\check{G}^{ad}}$. 
\subsubsection{}
We let $\mathscr{I}$ denote the set of nodes of the
Dynkin diagram of $G$.
For $i \in \mathscr{I}$, the corresponding 
simple roots and coroots are denoted 
\[
\alpha_i \in \Lambda_G \quad \text{and} \quad \halpha_i \in \check{\Lambda}_G. 
\] 

\subsubsection{}\label{ss:involution}

Our choice of pinning defines a standard
involutive anti-homomorphism\footnote{One sometimes
	finds this involution called the Chevalley involution, or the \emph{Cartan} 
	involution,
	but the latter terminology is potentially misleading.}  \[\tau:\fg 
	\xar{\sim} \fg.\]
More precisely, for $i \in \mathscr{I}$, let
$e_i \in \fn$ be the unique vector of weight $\alpha_i$
such that $\psi(e_i) = 1$. 
Then $\tau$ is the unique involution such that
$\tau([x,y]) = -[\tau(x),\tau(y)]$,
$\tau|_{\ft} = \id_{\ft}$, and
\[[e_i,\tau(e_i)] = \check{\alpha_i}, \quad \quad \text{for } i \in 
\mathscr{I}.\]

We observe that $\tau$ lifts to 
an involutive anti-homomorphism on $G$,
which we also denote by $\tau$.

\subsection{Loops and arcs}

\subsubsection{}

For any affine variety $Z$ of finite type, 
we let $\fL Z$ denote its algebraic loop space and
$\fL^+ Z$ denote its algebraic arc space. The former is
an ind-scheme, while the latter is an affine scheme. 
There is a canonical evaluation map $\fL^+ Z \to Z$ given by
evaluation of a jet at the origin.

For $Z = H$ an affine algebraic group, 
$\fL H$ is a group ind-scheme, with $\fL^+H \subset \fL H$ 
a group subscheme.
\subsubsection{}

The Borel subgroup $B \subset G$ defines the \emph{Iwahori}
subgroup \[I := \fL^+ G \times_G B \subset \fL^+ G \subset \fL G.\]
Dually, we a preferred Iwahori subgroup 
$\check{I} \subset \fL \check{G}$. We denote the prounipotent radical of $I$ by 
$\mathring{I}$ and note the canonical isomorphism 
\[
T \simeq I / \mathring{I}.
\]

\subsection{Weyl groups} 

The combinatorics of affine Weyl groups plays an important
role in this paper. We recall notation and fundamental
constructions below.

\subsubsection{}

We let $W_f$ denote the Weyl group of $G$, and remind that
$W_f$ is also the Weyl group of $\check{G}$. 

\subsubsection{}

The extended affine Weyl group of 
$G$ is the semidirect product 
\[ \widetilde{W} \coloneqq W_f \ltimes \check{\Lambda}_G.\]
The subgroup of $\widetilde{W}$ given by 
\[
W \coloneqq W_f \ltimes \check{Q} 
\] 
\noindent is the affine Weyl group, where we remind
that $\check{Q}$ was the coroot lattice.

\subsubsection{}

Let $S_f \subset W_f$ denote the set of simple reflections
$s_i$ for $i \in \mathscr{I}$. 
Let $S \subset W$ denote the union of $S_f$ with 
the set of simple affine reflections as in 
Section 7 of \cite{kitty}. We remind that the
simple affine reflections are indexed by simple
factors of $G$.

The pairs $(W_f,S_f)$ and $(W,S)$ are {Coxeter systems},
i.e., Coxeter groups with preferred choices of simple
reflections.
We remind that the Bruhat order and length function on $W$ 
each extend in a standard way to $\widetilde{W}$. 

%

\subsection{Categories} 

We repeatedly work with DG categories and their 
symmetries. In this setting, we use the following conventions.

\subsubsection{}

We let $\DGCat_{cont}$ denote the symmetric monoidal
$\infty$-category of cocomplete DG categories and
continuous DG functors as 
defined in \cite{grbook} Section I.1.10.
We denote the binary product underlying the symmetric
monoidal structure by $-\otimes -$; this is the
\emph{Lurie tensor product} of \cite{higheralgebra}.

For simplicity, we sometimes refer to $\infty$-categories
as \emph{categories}, and similarly for DG categories.

\subsubsection{}

Given a $t$-structure on a DG category $\cC$, 
we let $\cC^{\leqslant 0}$ and  $\cC^{\geqslant 0}$ denote
the subcategories of connective and coconnective
objects. That is, we use cohomological indexing notations. 

We denote the heart of such a $t$-structure by 
\[\cC^\heartsuit := \cC^{\leqslant 0} \cap \cC^{\geqslant 0}.\]

\subsection{$D$-modules} 

We make essential use of categories of $D$-modules on 
ind-pro-finite-type schemes such as $\fL G$, as developed in \cite{beraldo} 
and \cite{rdm}. For an indscheme $X$, we denote by $D(X)$ what is in {\em loc. 
	cit.} denoted by 
$D^*(X)$. 

\subsubsection{}
By functoriality, $D(\fL G)$ carries a canonical convolution 
monoidal structure. We denote the corresponding 
($\infty$-)category of DG categories 
equipped with an action of $D(\fL G)$ by 
\[
D(\fL G)\mod \coloneqq D(\fL G)\mod(\DGCat_{cont}).
\] 
\noindent We use similar notation for other group  
(ind-)schemes 
such as $\check{I}$ and $\fL N$.

\subsection{Invariants and coinvariants} 

\subsubsection{}

Given a group ind-scheme $H$, and a $D(H)$-module 
$\cC$, we denote its categories of invariants and coinvariants 
respectively by 
\[  \cC^H \coloneqq
\mathsf{Hom}_{D(H)\mod}(\Vect,\cC)
 \quad  \text{and} \quad  
 \cC_H \coloneqq \Vect \underset{D(H)}{\otimes} \cC.
\]
See \cite{beraldo} for further discussion. 

\subsubsection{}

Similarly, for a multiplicative $D$-module $\chi$ on $H$, 
we denote the 
corresponding categories of twisted invariants and twisted coinvariants by 
\[ \cC^{H, \chi} \quad  \text{and} \quad  \cC_{H, \chi}.
\]
Our multiplicative $D$-modules will be obtained
by one of the following two procedures.
\subsubsection{}

First, any character \[\lambda \in 
\Hom(H,\mathbb{G}_m) \otimes_{\mathbb{Z}} \mathbb{C}\] 
determines a character $D$-module $``t^\lambda"$ of $H$.
In this case, we denote twisted
invariants by $\cC^{H,\lambda}$.
We apply this construction particularly for $H = T$ and $H = I$.

\subsubsection{}
Similarly, given an additive character
\[
\psi: H \rightarrow \mathbb{G}_a,
\]
we obtain a character $D$-module 
$``e^\psi"$ on $H$. Here we denote twisted
invariants by $\cC^{H,\psi}$.
We apply this constructions for $H = \fL N$. 

\subsubsection{}

Suppose $H$ is an affine group scheme with 
pro-unipotent radical $H^u$. We suppose that
$H/H^u$ is finite type.

By \cite{beraldo}, the canonical 
forgetful map $\cC^H \rightarrow \cC$ admits a 
continuous right adjoint 
$\on{Av}_*^H$.
Moreover, there is a canonical
equivalence $\cC_H \simeq \cC^H$ fitting into a commutative
diagram
\begin{equation} \label{e:inv1}
\vcenter{\xymatrix{
\cC \ar[d] \ar[dr]^{\on{Av}^H_*} & \\
\cC_{H} \ar[r] & \cC^{H}.
}}
\end{equation}
\noindent The same applies in the presence of a multiplicative
$D$-module $\chi$.

\subsubsection{}

Now suppose $\cC \in D(\fL G)\mod$. 
Let $\psi:\fL N \to \mathbb{G}_a$ denote the Whittaker character of 
$\fL N$. 

In this case, \cite{whit} Theorem 2.1.1 provides
a canonical equivalence
\begin{equation}\label{e:inv2} 
\cC_{\fL N, \psi} \simeq \cC^{\fL N, \psi}.
\end{equation}
\noindent We highlight that this case is more subtle than
that of an affine group scheme considered above. 

\subsubsection{}\label{sss:quot-not}

As a final piece of notation, for $X$ an ind-scheme with 
an action of a group ind-scheme $H$,
we use the notation
\begin{equation} \label{e:notinv}
D(X/H, \chi) \coloneqq D(X)_{H, \chi}.
\end{equation}
\noindent By \cite{rdm} Proposition 6.7.1, this
notation is unambiguous. 

In the setting of either \eqref{e:inv1} or \eqref{e:inv2},
we remark that these coinvariants coincide with invariants. 

\subsection{Levels} \label{ss:levels}

Recall that a {level} $\kappa$ for $G$ is a $G$-invariant
symmetric bilinear form \[\kappa:\Sym^2(\fg) \to \mathbb{C}.\]

\subsubsection{}
We let $\kappa_{\fg,c}$ denote the 
critical level for $G$, i.e., $-\frac{1}{2}$ times the Killing form
of $G$. Where $G$ is unambiguous, we simply write $\kappa_c$.

\subsubsection{}
Suppose $G$ is simple. A level $\kappa$ is rational 
if $\kappa$ is a rational multiple of the Killing form
and irrational otherwise.
We say $\kappa$ is positive if
$\kappa-\kappa_c$ is a positive rational multiple
of the Killing form. We say a level $\kappa$ is 
negative
if $\kappa$ is not positive or critical.
In particular, any irrational level is negative.

For general reductive $G$, we say a level $\kappa$
is rational, irrational, positive,
or negative if its restrictions to each simple factor
are so.

\subsubsection{}
A level $\kappa$ is nondegenerate if $\kappa-\kappa_c$ is 
nondegenerate as a bilinear form. 
For such $\kappa$, the dual level
$\check{\kappa}$ for $\check{G}$ is the unique nondegenerate level
such that the restriction of $\check{\kappa}-\check{\kappa}_{\check{\fg},c}$ 
to $\ft^*$
and the restriction of $\kappa-\kappa_{\fg,c}$ to ${\ft}$ are dual symmetric 
bilinear forms.

\subsubsection{}
For a simple Lie algebra $\mathfrak{g}$, the
basic level $\kappa_{\fg,b} = \kappa_b$ 
is the unique positive level such that the short coroots have squared length 
two, i.e. \[\min_{i \in \mathscr{I}} \kappa_b(\halpha_i,\halpha_i) = 2. \]

\subsection{Twisted $D$-modules} 

Given a level $\kappa$, there is a canonical 
monoidal DG category of 
twisted $D$-modules, $D_\kappa(\fL G)$, see for example
\cite{whit} Section 1.29. We again use the notation
\[ 
D_\kappa(\fL G)\mod \coloneqq D_\kappa(\fL G)\mod(\DGCat_{cont}).
\]

\subsubsection{} We recall that the multiplicative twisting defined by $\kappa$ 
is 
canonically trivialized on $\fL^+ G$ and 
$\fL N$. In particular, for \begin{equation} \label{e:object}\cC \in 
D_\kappa(\fL G)\mod,\end{equation}
we can make sense of invariants and coinvariants of
$\fL^+ G$ and $\fL N$ with coefficients in $\cC$.
The same applies for any subgroup, in particular for
$I \subset \fL^+ G$. Also, the same applies with a twisting,
e.g. in the Whittaker setup for $\fL N$. We remark that the identification 
\eqref{e:inv2}
of Whittaker invariants and
coinvariants \cite{whit} was proved more generally
for $D_{\kappa}(\fL G)\mod$.

\begin{example}

$D_\kappa(\fL G)$ carries commuting actions of 
$D(\fL N)$ and $D(I)$, so following the convention of
Section \ref{sss:quot-not}, we use the notation
\[
D_\kappa(\fL N, \psi \backslash \fL G / I) 
\]
\noindent for the appropriate invariants = coinvariants
category.

\end{example}

\subsection{Affine Lie algebras} 

\subsubsection{}

Let $\fL \fg$ denote the 
Lie algebra of $\fL G$, considered with its natural
inverse limit topology, i.e. \[\fL \fg = \fg\otimes \mathbb{C}(\!(t)\!) = 
\lim_n \fg \otimes \mathbb{C}(\!(t)\!)/t^n \mathbb{C}[\![t]\!].\]
 Given a level $\kappa$, one obtains a continuous 2-cocycle $\fL \fg \otimes 
 \fL 
 \fg \to \mathbb{C}$ given by 
\[
\begin{gathered}
\xi_1 \otimes \xi_2 \mapsto 
\Res \kappa(\xi_1,d\xi_2).
\end{gathered}
\]

\noindent Here $d$ is the exterior derivative
and $\Res$ is the residue. We denote the corresponding central extension by
\[ 0 \rightarrow \mathbb{C} \mathbf{1} \rightarrow \gk \rightarrow \fL 
\fg 
\rightarrow 0.
\]

\subsubsection{}

We denote by $\gk\mod^{\heartsuit}$ the abelian category of 
smooth representations of $\gk$
on which the central element $\mathbf{1}$ acts via the identity. 

We let $\gk\mod$ denote the DG category introduced by
Frenkel--Gaitsgory in Sections 22 and 23 of \cite{fg}. 
This DG category is compactly generated and carries a canonical
$t$-structure with heart $\gk\mod^{\heartsuit}$. However,
there are some non-zero objects in \[\gk\mod^{-\infty} \coloneqq 
\bigcap_n \hspace{1mm} \gk\mod^{\leqslant -n},\]so $\gk\mod$ is not the derived
category of $\gk\mod^{\heartsuit}$.

\subsubsection{}

In Sections 10 and 11 of \cite{mys}, a $D_{\kappa}(\fL G)$-module
structure on $\gk\mod$ was constructed, 
enhancing previous constructions of Beilinson--Drinfeld
and Frenkel--Gaitsgory, cf. \cite{bdh} Section 7 and
\cite{fg2} Section 22.

\subsubsection{} Let $H$ be a sub-group scheme of $\fL^+ G$ of finite 
codimension, and consider the corresponding category of equivariant objects
\[
 \gk\mod^{H}. 
\]
By \cite{whit} Lemma A.35.1, the bounded below category
\[ \gk\mod^{H,+}\]
\noindent canonically identifies with bounded below derived category of its 
heart
$\gk\mod^{H,\heartsuit}$, which are Harish--Chandra modules for the pair $(\gk, 
H)$. Moreover, $\gk\mod^H$ is compactly generated by inductions of finite 
dimensional $H$-modules.

\subsubsection{}
\label{sss:dotaction}
Let $\kappa$ be a level. 
We abuse notation in letting $\kappa$ also denote the map
$\kappa:\ft \to \ft^*$.
 
The dot action of
$\widetilde{W}$ on $\ft^*$ is defined by having $W_f$ act through the usual dot 
action, and $\check{\Lambda}_G$ act through translations via $\kappa - 
\kappa_c$. I.e., writing $\rho$ for the half sum of the positive roots, for any 
$\lambda \in \ft^*$ we have
\[
\begin{gathered}
w \cdot \lambda \coloneqq w(\lambda+\rho)-\rho, 
\quad \quad w \in W_f \subset
\widetilde{W},  \\
\check{\mu} \cdot \lambda \coloneqq 
\lambda+(\kappa-\kappa_c)(\check{\mu}), 
\quad\quad  \check{\mu} \in \check{\Lambda} \subset 
\widetilde{W}.
\end{gathered}
\]
\subsubsection{} To discuss integral Weyl groups, we need some more standard 
facts about this dot action. Recall that the affine real coroots of $\gk$ are a 
subset of 
$\ft \oplus \mathbb{C} \mathbf{1}$. In particular, to a such a coroot 
$\halpha$, we may associate its classical part $\halpha_{cl}$, i.e. its 
projection to $\ft$. This is a coroot of $\fg$, and in particular we may 
associate a classical root $\alpha_{cl} \in \ft^*$.

\subsubsection{}\label{sss:affcoroots} Via a standard construction, $\halpha$ 
acts as an affine 
linear functional on $\ft^*$, which we denote by $\langle \halpha, -\rangle$,  
in such a way that, 
writing $s_{\halpha}$ for the associated reflection in $\widetilde{W}$, one has 
\[
 s_{\halpha} \cdot \lambda = \lambda - \langle\halpha, \lambda + \rho \rangle 
 \alpha_{cl}, \quad \quad \lambda \in \ft^*. 
\]
Briefly, this arises via restricting a linear action of $\widetilde{W}$ on 
$(\ft \oplus \mathbb{C} \mathbf{1})^*$ to the affine hyperplane of functionals 
whose pairing with $\mathbf{1}$ is 1. We refer the reader to 
Sections 3.1 and 
3.4 of \cite{lpw}, where this is reviewed in greater detail.

\section{Fundamental local equivalences} \label{s:fle}

In this section, we suppose $G$ is simple of adjoint type
and $\kappa$ is a negative level for $G$. Under these assumptions,
we prove the main theorem, i.e., the fundamental local
equivalence for good $\kappa$. In Appendix \ref{a:reduc}, we deduce the 
same result for general $G$ and general good $\kappa$ from this case.

\subsection{Overview of the argument}\label{ss:overview}

The proof of the main theorem requires fine arguments involving
combinatorics of affine Lie algebras. To help the reader understand what
follows, we begin with an overview
of the main ideas. 
This inherently requires referring to concepts that have not been
introduced yet, so the reader may safely skip this material and
refer back as necessary.

We omit some technical considerations at this point
in the discussion. For example, we do not carefully
distinguish here between abelian and derived
categories.

\subsubsection{}

Suppose $\lambda \in \ft^{*}$ is a weight of $\fg$.
Let us denote the block of Category $\OO$ for $\gk$ containing the Verma module 
$M_\lambda$ by 
\[\OO_{\kappa,\lambda} \subset 
\widehat{\fg}_{\kappa}\mod^{\mathring{I}}.\]
\noindent In Section \ref{sss:intweylgp}, we recall that $\lambda$ determines a 
subgroup
$W_{\lambda} \subset W$, its integral Weyl group.\footnote{In spite of 
the notation, $W_{\lambda}$ depends also on $\kappa$.} This is a 
Coxeter group, i.e. comes equipped with a set of simple reflections.

We make important use of Theorem \ref{t:fiebig}, which is due to 
Fiebig \cite{fiebig}, following earlier
work of Soergel \cite{soergel}. 
This result asserts that 
if $\lambda$ is antidominant\footnote{In the sense of affine Kac--Moody
algebras. In particular, the definition depends on $\kappa$.} 
$\OO_{\kappa,\lambda}$ is 
determined as a category by the data of (i) the Coxeter group $W_{\lambda}$ 
along with (ii) the subgroup 
\begin{equation}
\label{e:sfd} W_\lambda^\circ \subset W_\lambda\end{equation} 
stabilizing $\lambda$ under the 
dot action of $W$ on $\ft^*$, cf. 
Section \ref{sss:dotaction}.

\begin{re}

For us, the most important case is when $\lambda$ is integral. 
Here we write $W_{\fg,\kappa}$
in place of $W_{\lambda}$. 
In this case, $W_{\fg,\kappa}$ contains
the finite Weyl group $W_f$. 
If $\kappa$ is irrational, the two are equal, and 
the simple reflections in $W_{\fg,\kappa} = W_f$ 
are the usual ones determined by our fixed Borel.
If $\kappa$ is rational, there is one additional simple reflection;
we provide an explicit formula for it in 
Lemma \ref{l:simprefs}.

\end{re}

\subsubsection{}

In Theorem \ref{t:whit-blocks}, we find a block decomposition 
for the Whittaker category $\Whit_{\kappa}^{\aff}$ 
on the affine flag variety for $G$.
These blocks are indexed by $W_{\fg,\kappa}$-orbits in 
$\check{\Lambda}_G = \Lambda_{\check{G}}$. 

In Theorem \ref{t:whit-km-neut}, we show that up to varying
$\kappa$ by an integral translate (which does not affect the Whittaker
category), its neutral block is equivalent to an integral block of
Category $\OO$ for $\widehat{\fg}_{-\kappa}$. 
We generalize this to general blocks at good\footnote{In our actual exposition, 
the definition of good level is essentially rigged so such a result holds. 
The content is rather in Proposition \ref{p:adapt}, which provides a concrete 
description of good levels.} 
levels in Corollary \ref{c:wtokm}. These identifications preserve
the natural highest weight structures on both sides. From now on, we assume 
that $\kappa$ is good.

\subsubsection{}

For each block of $\Whit_{\kappa}^{\aff}$, we explicitly compute
the combinatorial datum \eqref{e:sfd} of the corresponding block of 
Category $\OO$ for the corresponding Kac--Moody 
algebra\footnote{Again,
this Kac--Moody algebra is essentially $\widehat{\fg}_{-\kappa}$, 
except that we may need to replace $-\kappa$ by an integral translate.}
and provide a form of Langlands duality for this datum. 

Essentially by construction, for any such block the corresponding integral 
Weyl group is $W_{\fg,\kappa}$, with simple
reflections as indicated above. In Theorem \ref{t:intw}, we construct
an isomorphism $W_{\fg,\kappa} \simeq W_{\check{\fg},\check{\kappa}}$
preserving simple reflections, and 
that is the identity on $W_f$.\footnote{We are not aware of the identification
of Theorem \ref{t:intw} having appeared previously in the literature.} 

Note that for any block of $\OO$ for the Kac--Moody algebra, the
corresponding integral Weyl group canonically acts on the set
of isomorphism classes of simple objects in this block. Therefore,
the above considerations provide an action of $W_{\check{\fg},\check{\kappa}}$
on the set of isomorphism classes of simple objects in 
$\Whit_{\kappa}^{\aff}$, which is canonically
identified with $\check{\Lambda}_G = \Lambda_{\check{G}}$. 
Again, by construction, this action coincides with the dot action of
Section \ref{sss:dotaction}, but for $\check{G}$ rather than $G$.


In Corollary \ref{c:minel} and in the proof of 
Theorem \ref{t:fle-tame},
we check that simple objects of 
$\Whit_{\kappa}^{\aff}$ corresponding to
antidominant weights of $\check{G}$ are also standard objects for
the highest weight structure on this category. 
These observations amount to 
matching the data \eqref{e:sfd} with that of affine Category $\OO$ for
$\check{G}$, completing the proof of Theorem \ref{t:main} in
the Iwahori case. 

\subsubsection{}

In Section \ref{ss:parahoric}, under the hypotheses of this
section, we deduce the parahoric version of
Theorem \ref{t:main} from the Iwahori version.
In particular, this includes the spherical version of the theorem. 

We do this by identifying the parahoric categories as full\footnote{This 
fully faithfulness only holds at the abelian categorical level.}
subcategories of the corresponding Iwahori categories and then identifying the
essential images under the isomorphism of Theorem \ref{t:fle-tame}.

\subsection{Block decomposition for the Whittaker category}

We begin by decomposing the Whittaker category into 
blocks.  

\subsubsection{}To do so, it will be useful to simultaneously consider the case 
of twisted Whittaker categories. Thus, we fix $\lambda \in \ft^*$, and 
consider

\[
\Whit_{\lambda}^{\aff} \coloneqq 
D_{\kappa}(\fL N,\psi \backslash \fL G / I,-\lambda).
\]
While these DG categories depends on both $\lambda$ and $\kappa$, we will study 
them with fixed $\kappa$ and varying $\lambda$, and for this reason suppress 
$\kappa$ for ease of notation. 


\subsubsection{}\label{sss:indexing}

For indexing reasons, it will be convenient to rewrite this category  
as follows. Consider the automorphism of $\fL G$ given by 
\[
g \mapsto \tau(g) \check{\rho}(t^{-1}),
\]
where $\tau$ is as in Section \ref{ss:involution}. This induces an equivalence
\begin{equation} \label{e:condone}
D_\kappa(\fL N, \psi \backslash \fL G / I,-\lambda) \simeq 
D_{-\kappa}(I,\lambda \backslash \fL 
G /\fL N^-, \psi)
\end{equation}
where the $\psi$ on the right-hand side of \eqref{e:condone} denotes a 
nondegenerate character of $\fL N^-$ of 
conductor one. In the 
following discussion, we will think of $\Whit_{\lambda}^{\aff}$
via the latter expression.

\subsubsection{} \label{sss:intweylgp}It will be important for us to consider 
$\Whit_{\lambda}^{\aff}$ as a module for the following version of the
affine Hecke category. 

Let $W_\lambda$ denote the integral Weyl group of $\lambda$.\footnote{Later in 
the paper, the twist will be fixed to $\lambda = 0$, 
but the group and level will vary. We will accordingly denote this integral 
Weyl group by $W_{\fg, -\kappa}$ instead. We hope this does not cause 
confusion, and will reintroduce this change in notation when it first occurs.} 
Recall 
this 
is the subgroup of $W$ generated the reflections $s_{\halpha}$ 
corresponding to affine coroots $\halpha$ satisfying \[\langle \halpha, \lambda 
\rangle \in 
\mathbb{Z},\]where the pairing is via the level $-\kappa$ action as in 
Section \ref{sss:affcoroots}. For $w \in W_\lambda$, write $j_{w, !*}$ for the 
intermediate 
extension of the 
simple object of 
\[
D_{-\kappa}(I, \lambda \backslash IwI / I, -\lambda)^\heartsuit.
\]
We define $\sH_\lambda$ to be the full subcategory 
\begin{equation} \label{e:heckealg} 
D_{-\kappa}(I, \lambda \backslash \fL G / 
I, -\lambda)
\end{equation}
generated under colimits and shifts by the objects
\[ j_{w, !*}, \quad \text{for} \quad  w \in W_\lambda.\]

By Section 4 of \cite{ly},\footnote{While {\em loc. cit.} is written ostensibly 
over a finite field and in the finite type setting, the arguments we cite from
it straightforwardly adapt to the present situation.} $\sH_\lambda$ is closed 
under convolution,
so admits a unique monoidal DG structure for which its embedding
into $\eqref{e:heckealg}$ is monoidal. In addition, \emph{loc. cit}. 
shows that $\sH$ is the neutral block of 
$\eqref{e:heckealg}$, i.e., it is the minimal direct summand containing the 
identity element.

\subsubsection{}
\label{sss:iwd}

Recall that the Iwasawa decomposition provides $\fL G$ with
a stratification by the double cosets
\[I w \fL N^-, \quad \quad \text{for}  \quad w \in \widetilde{W}.\] 
Let $\widetilde{W}^f \subset \widetilde{W}$ denote the subset
of elements of minimal length in their left $W_f$-cosets. We now describe the 
affine Whittaker category for a single stratum.

\begin{lemma} For $w \notin \widetilde{W}^f$, the affine Whittaker category 
on 
$Iw \fL N^-$ vanishes, i.e. we have %
	\[ D_{-\kappa}(I, \lambda \backslash I w \fL N^-/\fL N^-, \psi) \simeq 0.
	\]
	For $w \in \widetilde{W}^f$, $!$-restriction to any closed point gives an 
	equivalence
	\[
	D_{-\kappa}(I, \lambda \backslash I w \fL N^- / \fL N^-, \psi) \simeq \Vect.
	\]
\end{lemma}

\begin{proof}The twistings corresponding to both $\kappa$ and $\lambda$
	are trivializable on a double coset. Moreover, $\psi$ is trivial
	on the stabilizer in $\fL N^-$ of the point
	$w \in I\backslash\fL G$ if and only if $w \in \widetilde{W}^f$, which 
	implies the desired identities. 
\end{proof}

%
%

For $w \in \widetilde{W}^f$ we denote the corresponding standard, simple, 
and 
costandard 
objects of $\Whit_{\lambda}^{\aff}$ by 
\[ j_{w, !}^\psi, \quad  
j_{w, !*}^\psi, \quad \text{and} \quad j_{w, *}^\psi.\]
Explicitly, under the identification with 
$\Vect$ above, they correspond to the relevant
extensions of $\mathbb{C}[-\ell(ww_{\circ})] \in \Vect$,
where $\ell$ denotes the length function on $\widetilde{W}$, and $w_\circ$ 
denotes the 
longest element of $W_f$. 

\subsubsection{}We next obtain the block decomposition of 
$\Whit_{\lambda}^{\aff}$. To state it, for any double coset $$W_\lambda y 
W_f$$we write 
$\Whit_{\lambda,y}^{\aff}$ 
for the full subcategory of $\Whit_{\lambda}^{\aff}$ generated under colimits 
and shifts by the objects 
\[
j_{x, *}^\psi, \quad \quad \text{for } x \in W_\lambda y W_f 
\hspace{1mm}\cap \hspace{1mm} 
{\widetilde{W}^f}.
\]

\begin{theo}\label{t:whit-blocks} Each $\Whit_{\lambda, y}^{\aff}$ is preserved 
by the 
	action of $\sH_\lambda$, and the direct sum of inclusions yields an 
	$\sH_\lambda$-equivariant equivalence 
	\begin{equation} \label{e:blockwhit}
	\underset{ y \in W_\lambda \backslash \widetilde{W} / W_f }
	{\bigoplus} \Whit_{\lambda, y}^{\aff} 
	\to \Whit_\lambda^{\aff}.
	\end{equation}
\end{theo}

\begin{proof}
	
	For $w \in \widetilde{W}$ consider the corresponding standard, simple, and 
	costandard objects 
	\[ j_{w, !}, \quad  j_{w, !*}, \quad \text{and} \quad  j_{w, *} \quad 
	\text{of} \quad D_{-\kappa}(I, w \cdot \lambda \backslash 
	\fL G / I, -\lambda)^{\heartsuit}.\]
	\noindent We use $-\overset{I}{\star}-$ to denote the convolution functor
	\[ D_{-\kappa}(\fL G / I, -\lambda) \otimes D_{-\kappa}(I, \lambda 
	\backslash \fL G) \rightarrow D_{-\kappa}(\fL G),\]
	and the induced functor
	\begin{equation}
	D_{-\kappa}(I, w \cdot \lambda \backslash \fL G / I, -\lambda ) \otimes 
	D_{-\kappa}(I, \lambda \backslash \fL G / \fL N^-, \psi) \rightarrow 
	D_{-\kappa}(I, w \cdot \lambda \backslash \fL G / \fL N^-, \psi).
	\label{e:convrat}
	\end{equation}
	
	For any $w \in \widetilde{W}$, if we by abuse of notation denote its image 
	in 
	$\widetilde{W}^f \xrightarrow{\sim} \widetilde{W}/W_f$ again by $w$, we 
	claim 
	there exist equivalences
	\begin{equation}\label{e:convid} 
	j_{w, !} \overset{I}{\star} j_{e, *}^\psi \simeq j_{w, 
		!}^\psi \quad \quad 
	\text{and} 
	\quad \quad j_{w, *} 
	\overset{I}{\star} j_{e, *}^\psi \simeq j_{w, *}^\psi. 
	\end{equation}
	Note that in \eqref{e:convid}, the object $j_{e, *}^\psi$ belongs to 
	$\Whit_{\lambda}^{\aff}$, whereas $j_{w, !}^\psi$ and $j_{w, *}^\psi$ 
	belong to $\Whit_{w \cdot \lambda}^{\aff}$.

	We break the proof \eqref{e:convid} into several cases. First, suppose $w$ 
	lies in 
	$\widetilde{W}^f$. 
	Then the second identity in \eqref{e:convid} 
	follows from the observation that the convolution map 
	\begin{equation}
	\label{e:convmap}
	IwI \overset{I}{\times} I\fL N^- \rightarrow Iw \fL N^-
	\end{equation}
	is an isomorphism. The first identity for such $w$ then follows by the 
	cleanness of $j_{e, *}^\psi$ and the ind-properness of the multiplication
	map
	\[
	\fL G \overset{I}{\times} \fL G \rightarrow \fL G.
	\]
	
	Next, suppose $w = s$ is a simple reflection of $W_f$. 
	The image of the convolution map as in \eqref{e:convmap} is contained in 
	the locally closed sub-indscheme of $\fL G$ corresponding to the strata 
	$I \fL N^- \cup I s \fL N^-$. As the only stratum of its closure which
	supports Whittaker sheaves is $I \fL N^-$, 
	it is enough to compute the $!$-restrictions of our
	convolutions to the identity 
	\[i: e \rightarrow \fL G.\]
	Let us write $\iota$ for the involution $g \mapsto g^{-1}$ on $\fL G$. By 
	base change we may compute the $!$-fibre as  
	\[
	i^!(j_{s,*} \overset{I}{\star} j_{e,*}^{\psi}) \simeq 
	\Gamma_{dR}(I \backslash \fL G, \iota_* j_{s, *} \overset{!}{\otimes} j_{e, 
	*}^\psi) \simeq
	\Gamma_{dR}(\mathbb{P}^1\setminus \{0,\infty\},``e^t"[-\ell(w_0)]) \simeq 
	\mathbb{C}[- \ell({w_\circ})],
	\]
	as desired. A similar calculation on $\mathbb{P}^1$ yields for
	the first identity in \eqref{e:convid} for $w = s$.

	Finally, for general $w \in \widetilde{W}$, we write
	$w = w^f w_f$ for $w^f \in \widetilde{W}^f$ and $w_f \in W_f$. 
	Choosing a 
	reduced expression for $w_f$, and the corresponding factorizations of 
	$j_{w, !}$ and $j_{w, *}$, we are reduced to the
	cases considered above for twists in $\widetilde{W} \cdot \lambda$. Given 
	\eqref{e:convid}, the assertions of the theorem follow by 
	the same arguments as in \cite{ly} Propositions 4.7 and 4.9. 
\end{proof}

%
%

\subsubsection{} Having obtained the block decomposition of 
$\Whit_{\lambda}^{\aff}$, we now record some properties of each block as an 
$\sH_\lambda$ module. We begin with some relevant combinatorics.

 Recall that 
$W_\lambda$ is a Coxeter group. Namely, write 
$\check{\Phi}^+$ for 
the positive real coroots, and consider the subset
\[\check{\Phi}^+_\lambda = \{ \halpha \in \check{\Phi}^+: \langle \halpha, 
\lambda \rangle \in \mathbb{Z} \}.
\] 
With this, a reflection $s_{\halpha}$ of $W_\lambda$ is simple if and only if 
$\halpha$ is not expressible as a 
sum of two other elements of $\check{\Phi}^+_\lambda$.

\begin{lem} \label{l:dcs}Each double coset $W_\lambda y {W}_{f}$ 
	contains a unique element minimal with
	respect to the Bruhat order.  In particular, each $W_\lambda$ 
	orbit on $\widetilde{W} / W_{f}$ contains a 
	unique element $yW_f$ minimal with respect to the Bruhat order. Moreover, 
	its stabilizer   
	\begin{equation}  y W_f y^{-1} \cap W_\lambda,\label{e:parsubgp}
	\end{equation}
	is a parabolic subgroup of $W_\lambda$. 
\end{lem}

\begin{proof} Let $y$ be any element of minimal length in $W_\lambda y 
	W_{f}$. We first claim that $y \leqslant w_\lambda y$ for every 
	$w_\lambda 
	\in W_\lambda$. Indeed, for any reflection $s_{\check{\alpha}}$ of 
	$W_\lambda$ 
	with 
	corresponding positive coroot $\check{\alpha}$, it follows from the minimal 
	length of $y$ that  \begin{equation}\label{e:posy}s_{\check{\alpha}} y > 
	y,\end{equation}i.e., that 
	$y^{-1}(\check{\alpha}) > 0$. 
	The claimed minimality now follows by induction on the length of element of 
	$W_\lambda$ with respect to its Coxeter generators. 
	
	We next show $y \leqslant w_\lambda y w_f$, for any $w_f$ in $W_f$. 
	However, it is clear that $y \leqslant ys$, for any simple reflection $s$ 
	of $W_f$. This 
	implies that for any $w \in \widetilde{W}$, $y \leqslant w$ if and only if 
	$y 
	\leqslant ws_j$. Applying this and a straightforward induction on the 
	length 
	to $w = w_\lambda y$ yields the claim. 
	
	It remains to show that \eqref{e:parsubgp} is a parabolic 
	subgroup 
	of $W_\lambda$. To see this, consider $W_{y^{-1} \lambda} = y^{-1} 
	W_\lambda 
	y$. The positivity property \eqref{e:posy} of $y$ implies that conjugation 
	by $y^{-1}$ defines an 
	isomorphism of Coxeter systems \[W_\lambda \simeq W_{y^{-1} \lambda},\] 
	i.e. 
	exchanges their simple reflections. It is therefore enough to see that $W_f 
	\cap W_{y^{-1} \lambda}$ is a parabolic 
	subgroup of $W_{y^{-1} 
		\lambda}$, which is straightforward, cf. the 
	proof of Lemma 8.8 of \cite{lpw}. 
	\end{proof}

\subsubsection{}Combining Lemma \ref{l:dcs} and Theorem \ref{t:whit-blocks} we 
obtain:

\begin{cor} \label{c:minel}For a double coset $W_\lambda y W_f$ with minimal 
element $y$, the 
corresponding object of $\Whit_{\lambda, y}^{\aff}$ is clean, i.e., %
	\[
	 j_{y, !}^\psi \simeq j_{y, !*}^\psi \simeq j_{y, *}^\psi.
	\]
\end{cor}

In the following proposition, we collect some properties of the action of 
$\sH_\lambda$ on $j_{y, !}^\psi$. To state them, let us denote 
the set of elements of $W_\lambda$ of minimal length in their left cosets with 
respect to the parabolic subgroup \eqref{e:parsubgp} by 
\[ W_\lambda^f.
\]
We additionally recall that $\Whit_{\lambda}^{\aff}$ carries a canonical 
$t$-structure. Namely, an 
object $\mathscr{M}$ is coconnective, i.e., lies in 
$\Whit_{\lambda}^{\aff, \geqslant 0}$, if and only if 
\[ \Hom( j_{w, !}^\psi, \mathscr{M}) \in \Vect^{\geqslant 0}, \quad \text{for } 
w \in \widetilde{W}^f.
\]  
Such a $t$-structure may be seen directly via gluing of $t$-structures stratum 
by stratum, and coincides with the general construction of Section 5.2 and 
Appendix B of \cite{whit}, cf. \emph{loc. cit}. Remark B.7.1.

\begin{pro} For $y$ as in Corollary \ref{c:minel}, and any $w$ in $W_\lambda,$ 
there are equivalences 
	\begin{equation}\label{e:moves1}
	j_{w, !} \overset{I}{\star} j_{y, *}^\psi \simeq j_{wy, !}^\psi \quad 
	\text{and} \quad 
	j_{w, *} \overset{I}{\star} j_{y, *}^\psi \simeq j_{wy, *}^\psi,
	\end{equation}
	 Moreover, if $w$ lies in 
	$W_\lambda^f$, then convolution with $j_{y, *}^\psi$ yields an 
	isomorphism of lines
	\begin{equation} \label{e:moves2}
	\on{Hom}_{D_{-\kappa}(I, \lambda \backslash \fL G / I, 
		-\lambda)^\heartsuit}(j_{w, !}, j_{w, *}) \simeq 
	\on{Hom}_{\Whit^{\aff}_{\lambda}}(j_{wy, !}^\psi, j_{wy, *}^\psi).
	\end{equation} 
\end{pro}

\begin{proof} Both assertions follow from \eqref{e:convid} by standard 
arguments. Namely, the proof of \cite{ly} Proposition 5.2  
yields \eqref{e:moves1}. It follows from \eqref{e:moves1} that convolution 
with $j_{y, *}^\psi$ is $t$-exact, which readily implies \eqref{e:moves2}.  
\end{proof}

%
%


%
%

\begin{re} One can show that the clean objects $j_{y, !*}^\psi$ constructed 
above are the only clean extensions in $\Whit_{\lambda}^{\aff}$. Since we do 
not use this fact, we only sketch the proof. Namely, 
write $\leqslant$ for the Bruhat order on $W_\lambda^f$. One can in fact show 
that for any $w, x \in W_\lambda^f$ the space of intertwining operators
\[
\Hom_{\Whit_{\lambda}^{\aff, \heartsuit}}( j_{wy, !}^\psi, j_{xy, !}^\psi)
\]
vanishes unless $w \leqslant x$, in which case it is one dimensional and 
consists of embeddings. This may be deduced from \eqref{e:moves1} and the 
analogous classification of intertwining operators between standard objects 
of $\sH_\lambda$. The latter may be proved directly or deduced via localization 
from the Kac--Kazhdan theorem on homomorphisms between Verma modules. 
\end{re}

\subsection{Whittaker--singular duality}

We now use the above noted properties of the $\sH_\lambda$-action to produce, 
under 
necessary hypotheses, equivalences between the neutral block of 
$\Whit_{\lambda}^{\aff}$ 
and a block of Category $\mathscr{O}$ for $\gmk$.

\subsubsection{} Consider the category of $(I, \lambda)$-integrable 
modules for $\gmk$, which we denote by 
\begin{equation} \label{e:kmrep}
 \widehat{\fg}_{-\kappa}\mod^{I, \lambda}.
\end{equation}
Recall this is compactly generated by the Verma modules $M_\theta,
\text{ for } \theta \in \lambda + \Lambda_G,$ and note the integral Weyl 
group of any such 
$\theta$ is $W_\lambda$.

\subsubsection{} We identify the desired blocks by matching their combinatorics 
to that of $\Whit_{e, \lambda}^{\aff}$ as follows. Suppose $M_\Theta$ is 
an irreducible Verma module in \eqref{e:kmrep} such 
that the stabilizer of $\Theta$ under the dot action of $W_\lambda$ is 
given by 
\[ W_\lambda \cap W_f.
\]
Consider the corresponding block of \eqref{e:kmrep}, i.e. the full subcategory 
compactly generated by $M_{\theta}$, for $\theta \in W_\lambda \cdot \Theta$, 
and denote it by 
\[\widehat{\fg}_{-\kappa}\mod^{I, \lambda}_\Theta.\]

\begin{theo}\label{t:whit-km-neut}

There is a canonical $\sH_\lambda$-equivariant and $t$-exact 
equivalence
\begin{equation} \label{e:whit2kac}
 \Whit_{e, \lambda}^{\aff} \simeq \gmk\mod^{I, \lambda}_\Theta.
\end{equation}
\end{theo}

\begin{re}\label{r:sf-1}

For emphasis: by definition of $\Theta$, the integral Weyl group of this block
is $W_{\lambda}$, and the stabilizer of $\Theta$ in $W_{\lambda}$ is
$W_{\lambda} \cap W_f$.

\end{re}

\begin{re}

The minus sign in front of the level is an artifact of 
Section \ref{sss:indexing}. As we will discuss in greater detail in Section 
\ref{sss:changeintlevel}, we potentially 
need to apply an integral
translation to $\kappa \rightsquigarrow \kappa+\kappa'$ 
before finding $\Theta$ as above; in practice, this translation 
in particular assures that
of making $-\kappa-\kappa'$ is negative.

\end{re}

\begin{example} Theorem \ref{t:whit-km-neut} is always applicable when the 
twist $\lambda$ is trivial. Namely, by an integral translation, we may assume 
$-\kappa$ is sufficiently negative, in which case we may take $\Theta$ to be  
$- 
\rho$. 
\end{example}

\begin{proof}[Proof of Theorem \ref{t:whit-km-neut}] 

To produce an $\sH_\lambda$-equivariant functor, we will 
construct a $D_{-\kappa}(\fL G)$-equivariant functor and then pass to 
$(I,\lambda)$-equivariant objects. After further projecting onto a block of 
\eqref{e:kmrep}, we will show this is the sought-for equivalence. 
		
We begin by constructing an $D_{-\kappa}(\fL G)$-equivariant functor
\[ \mathsf{F}: D_{-\kappa}(\fL G / \fL N^-, \psi) \rightarrow \gmk\mod.
\]
To do so, recall that for any $D_{-\kappa}(\fL G)$-module $\scc$, one has a 
canonical equivalence
\begin{equation} \label{e:whicorep} \mathsf{Hom}_{D_{-\kappa}(\fL G)\mod}( 
D_{-\kappa}(\fL G / \fL N^-, \psi), 
\scc) \simeq \scc^{\fL N^-, \psi}.
\end{equation}
Explicitly, if we write $\on{ins} \delta$ for the insertion of the delta 
function at the identity of $\fL G$ into Whittaker coinvariants, the 
equivalence \eqref{e:whicorep} is given by evaluation at $\on{ins} \delta$. 
Applying this to $\scc 
\simeq \gmk\mod$, we 
obtain via the affine Skryabin equivalence, cf. \cite{whit}, 
\begin{equation} \label{e:whitmapsout}
 \mathsf{Hom}_{D_{-\kappa}(\fL G)\mod}( D_{-\kappa}(\fL G / \fL N^-, 
 \psi), \gmk\mod) \simeq \gmk\mod^{\fL N^-, \psi} \simeq 
 \sW_{-\kappa}\mod.
\end{equation}
Therefore, to produce $\mathsf{F}$ we must specify a module for the 
$\sW_{-\kappa}$-algebra.

We do so as follows. Write $\on{Zhu}(\sW_{-\kappa})$ for the Zhu algebra of 
$\sW_{-\kappa}$, and $Z\fg$ for the 
center of the universal enveloping algebra of $\fg$. Let us normalize the 
identification \[\on{Zhu}(\sW_{-\kappa}) \simeq Z \fg\] as in \cite{lpw}. 
Writing 
$\chi(\Theta)$ for the character of $Z \fg$ corresponding to $\Theta$, i.e. via 
its action on Verma module for $\fg$ with highest weight $\Theta$, consider the 
associated local 
cohomology $Z\fg$-module $Ri_{\chi(\Theta)}^! Z \fg$. We will take $\mathsf{F}$ 
to be associated to the 
corresponding 
$\sW_{-\kappa}$-module 
\begin{equation} \label{e:thewmod}
\on{pind}_{\on{Zhu}(\sW_{-\kappa})}^{\sW_{-\kappa}} Ri_{\chi(\Theta)}^! Z \fg,
\end{equation}
where $\on{pind}$ denotes the standard induction from Zhu algebra modules to 
vertex algebra modules, i.e., the left adjoint to the functor of taking 
singular vectors. 

Passing to $(I, \lambda)$-equivariant objects, we obtain a $D_{-\kappa}(I, 
\lambda \backslash \fL G / I, -\lambda)$-equivariant functor
\[
\mathsf{F}: D_{-\kappa}( I, \lambda \backslash \fL G / \fL N^-, \psi) 
\rightarrow \gmk\mod^{I, \lambda}.
\]
Let us determine $\mathsf{F}(j_{e, *}^\psi)$. To do so, write $\mathbb{C}_\psi$ 
for the one dimensional representation of $\fn^-$ associated to the additive 
character $\psi$ of $\fL N^-$. As we will explain in more detail below, we then 
have
\begin{align}\label{e:line1}
 \on{Av}^{I, \lambda}_{*} \on{pind}_{\on{Zhu}(\sW_{-\kappa})}^{\sW_{-\kappa}} 
 Ri_{\chi(\Theta)}^! Z \fg  &\simeq \on{Av}^{I, \lambda}_{*} 
 \on{pind}_{\fg}^{\gmk} (Ri^!_{\chi(\Theta)} Z\fg \underset{Z \fg}{\otimes} 
 \on{ind}_{\fn^-}^{\fg} \mathbb{C}_\psi)\\ & \label{e:line2}\simeq 
 \on{pind}_{\fg}^{\gmk} \on{Av}^{B, \lambda}_{*} (Ri^!_{\chi(\Theta)} Z \fg 
 \underset{Z \fg}{\otimes} \on{ind}_{\fn^-}^{\fg} \mathbb{C}_\psi)
 \\ & \label{e:line3}\simeq \bigoplus_{ \xi \in (W_f \cdot \Theta) \cap \lambda 
 + \Lambda_G} 
 M_\xi \otimes \det( \fb^*[-1]),
\end{align}
To see the first identification \eqref{e:line1}, one may use that 
\eqref{e:thewmod} is an 
iterated extension of Verma modules, and in particular arises from the first 
step of the adolescent Whittaker filtration of $\sW_{-\kappa}\on{-mod}$, cf. 
\cite{whit}.

 For the second identification \eqref{e:line2}, if we write $K_1$ 
for the first 
congruence subgroup of $\fL^+ G$, one has a corresponding factorization
\[
\on{Av}_{*}^{I, \lambda} \simeq \on{Av}_{*}^{K_1, \lambda} \circ 
\on{Av}_{*}^{B, \lambda}.
\] 
The second identification then follows from the $K_1$-integrability of a 
parabolically induced module, the prounipotence of $K_1$, cf. Theorem 4.3.2 
of \cite{beraldo}, and the 
$D(G)$-equivariance of 
$\on{pind}_{\fg}^{\gmk}$. 

For the third identification \eqref{e:line3}, for any 
$\xi \in 
\lambda + \Lambda_G$, let us write $M_{\xi, \fg}$ for the corresponding Verma 
module for $\fg$.  One has by adjunction
\[
 \on{Hom}_{\fg\mod^{B, \lambda}}(M_{\xi, \fg},  \on{Av}^{B, \lambda}_{*} 
 Ri^!_{\chi(\Theta)} Z \fg 
 \underset{Z \fg}{\otimes} \on{ind}_{\fn^-}^{\fg} \mathbb{C}_\psi) \simeq 
 \on{Hom}_{\fg\mod}(M_{\xi, \fg},  Ri^!_{\chi(\Theta)} Z \fg 
 \underset{Z \fg}{\otimes} \on{ind}_{\fn^-}^{\fg} \mathbb{C}_\psi).
\]
The latter vanishes unless $\xi \in W_f \cdot \Theta$, in which 
case we continue 
\[ \simeq \on{Hom}_{\fg\mod}(M_{\xi, \fg}, \on{ind}_{\fn^-}^{\fg} 
\mathbb{C}_\psi) \simeq 
\on{Hom}_{\fb\mod}(\mathbb{C}_\xi, \on{Res}_{\fg}^{\fb} \on{ind}_{\fn^-}^{\fg} 
\mathbb{C}_\psi) \simeq \det(\fb^*[-1]),
\]
where in the last step one uses the canonical identification 
\[ U(\fb) \simeq \on{Res}_{\fg}^{\fb} \on{ind}_{\fn^-}^{\fg} \mathbb{C}_\psi. \]
By our assumption on $\Theta$, each $M_{\xi, \fg}$, for $\xi \in (W_f \cdot 
\Theta) \cap \lambda + \Lambda_G$, generates a block of $\fg\mod^{B, \lambda}$ 
equivalent to $\Vect$, hence we have shown 
\[ \on{Av}_{*}^{B, \lambda} Ri^!_{\chi(\Theta)} Z \fg \underset{Z \fg}{\otimes} 
\on{ind}_{\fn^-}^{\fg} \mathbb{C}_\psi \simeq \bigoplus_{ \xi \in (W_f \cdot 
\Theta) \cap \lambda + \Lambda_G} 
M_{\xi, \fg} \otimes \det(\fb^*[-1]).
\]
The identity \eqref{e:line3} then follows by applying $\on{pind}_{\fg}^{\gmk}$.

The projection of the sum of Verma modules \eqref{e:line3} onto $\gmk{\mod}^{I, 
\lambda}_\Theta$ picks out the summand \[M_\Theta \otimes \det(\fb^*[-1]),\] 
cf. 
Lemma 8.7 of 
\cite{lpw}. Accordingly, we consider the composition 
\begin{equation} \label{e:comp}
 D_{-\kappa}( I, \lambda \backslash \fL G / \fL N^-, \psi) \xrightarrow{ 
 \mathsf{F}\otimes \det( \fb[1]) [\dim N^-]} \gmk\msf{-mod}^{I, \lambda} 
 \rightarrow \gmk\msf{-mod}^{I, 
 \lambda}_\Theta,
\end{equation}
where the latter map is projection onto the block. Note that while the blocks 
of \eqref{e:kmrep} are in general not preserved by $D_{-\kappa}( I, \lambda 
\backslash \fL G / I, -\lambda)$, they are preserved by $\sH_\lambda$. In 
particular, by construction the composition \eqref{e:comp}, which we denote by 
$\Psi$, is an $\sH_\lambda$-equivariant functor which sends $j_{e, *}^\psi$ to 
$M_\Theta$. 

It remains to see that $\Psi$ is an equivalence. 
By \eqref{e:moves1} and $\sH_\lambda$-equivariance, we obtain identifications
\[
\Psi( j_{w, !}^\psi) \simeq \Psi(j_{w, !} \overset{I}{\star} j_{e, *}^\psi ) 
\simeq 
j_{w, !} \overset{I}{\star} \Psi(j_{e, *}^\psi) \simeq M_{w \cdot \Theta}, 
\quad \quad  \text{for} \quad w \in W_\lambda, 
\]
where the last identity is a standard consequence of localization. Similarly, 
if for $\theta \in \lambda + \Lambda$ we denote the contragredient 
dual of $M_\theta$ by $A_\theta$, we obtain identifications 
\[ \Psi(j_{w, *}^\psi) \simeq \Psi(j_{w, *} \overset{I}{\star} j_{e, *}^\psi) 
\simeq j_{w, *} \overset{I}{\star} \Psi( j_{e, *}^\psi) \simeq A_{w \cdot 
\Theta}, \quad \quad \text{for} \quad w \in W_\lambda.
\] 
It is therefore enough to check, for any $y, w \in W_\lambda$, the map
\[
\Psi: \Hom_{\Whit_{\lambda}^{\aff}}( j_{y, !}^\psi, j_{w, !}^\psi) \rightarrow 
\Hom_{\gmk\mod^{I, \lambda}}( M_{y \cdot \Theta}, A_{y \cdot \Theta})
\]
is an equivalence. This again follows from \eqref{e:moves2}, as desired.  
\end{proof}

\subsubsection{} While Theorem \ref{t:whit-blocks} realizes the neutral block 
of $\Whit_{\lambda}^{\aff}$, this may be applied to other blocks as follows. 
Fix a double coset \[y \in W_\lambda \backslash \widetilde{W} / W_f,\] with 
associated minimal element $y$ and block $\Whit_{\lambda, y}^{\aff}$.

\begin{pro} \label{p:neutr}Convolution with the clean object
	\[ j_{y, *} \in D_{-\kappa}(I, \lambda \backslash \fL G / I, 
	- y^{-1} \cdot \lambda)
	\]
yields a $t$-exact equivalence
\[ \Whit_{y^{-1} \cdot \lambda, e}^{\aff} \xrightarrow{\sim} \Whit_{\lambda, 
y}^{\aff}.
\]
\end{pro}

\begin{proof}
The argument for \cite{ly} Proposition 5.2 applies
\emph{mutatis mutandis}.
\end{proof}

\subsection{Application of Whittaker-singular duality and the classification of 
good levels}

\subsubsection{}\label{sss:changeintlevel}We would like to relate an 
arbitrary block  
$\Whit_{\lambda, y}^{\aff}$ of $\Whit_{\lambda}^{\aff}$ to Kac--Moody 
representations. Via Proposition  
\ref{p:neutr}, we should apply Theorem \ref{t:whit-blocks} to $\Whit_{y^{-1} 
\cdot 
\lambda, e}^{\aff}$. Therefore, we would like to 
produce a suitable Verma module in 
\begin{equation} \label{e:thatoneeq} \gmk\mod^{I, y^{-1} \cdot \lambda}.
\end{equation}
It will useful to expand the collection of available Verma modules. For example,
if $-\kappa$ is positive rational and the twist $y^{-1} \cdot 
\lambda$ is trivial, there are no irreducible Verma modules in 
\eqref{e:thatoneeq}. To address this, 
note that for any integral level $\kappa'$ for $G$ one has a 
tautological $t$-exact equivalence
\[ D_{-\kappa}(I, y^{-1} \cdot \lambda \backslash \fL G / \fL N^-, \psi) \simeq 
D_{-\kappa + 
\kappa'}(I, y^{-1} \cdot \lambda \backslash \fL G / \fL N^-, \psi).
\]
We may further increase our supply of integral levels and characters as 
follows. Write $G_s$ for the simply 
connected form of $G$, and $I_s$ for the 
Iwahori subgroup of its loop group. Then the tautological embedding
\[ D_{-\kappa}( I_s, y^{-1} \cdot \lambda \backslash \fL G_s / \fL N^-, \psi) 
\rightarrow 
D_{-\kappa}(I, y^{-1} \cdot \lambda \backslash \fL G / \fL N^-, \psi)
\]
induces an equivalence of neutral blocks.

\subsubsection{}To analyze when we may find the desired Verma modules, we will 
need some basic properties of the relevant integral Weyl group. So, for an 
arbitrary level $\kappa_{\circ}$, let us denote by $W_{\fg, \kappa_{\circ}}$ 
the integral Weyl group of $0 \in \ft^*$ at level $\kappa_\circ$, i.e. what was 
denoted by
$W_0$ in the notation of \ref{sss:intweylgp}.

To describe the simple reflections in $W_{\fg, \kappa_{\circ}}$, recall the 
canonical identification of $W$ with the semidirect product
\[ W_f \ltimes \check{Q}.
\]
For any finite coroot $\halpha$ we denote by $s_{\halpha}$ the corresponding 
reflection in $W_f$, and for an element $\check{\lambda}$ of the coroot lattice 
$\check{Q}$ we write $t^{\check{\lambda}}$ for the corresponding translation in 
$W$. In addition, let us write $\check{\theta}_s$ for the short dominant 
coroot, $\check{\theta}_l$ for the long dominant coroot, and $r$ for the lacing 
number of $\fg$.  

\begin{lemma} \label{l:simprefs}For any $\kappa_{\circ}$ and integral level 
$\kappa'$, there are 
canonical identifications
	\begin{equation} \label{e:coxeq1}
	W_{\fg, \kappa_{\circ}} \simeq W_{\fg, -\kappa_{\circ}} \simeq W_{\fg, 
	\kappa_{\circ} + \kappa'}
	\end{equation}
	intertwining their inclusions into $W$. If $\kappa_{\circ}$ is irrational, 
	then $W_{\fg, \kappa_{\circ}} 
\simeq W_f$, i.e., they coincide as subgroups of $W$. If $\kappa_{\circ}$ is 
rational, write it as  
\[ \kappa_{\circ} = (-h^\vee + \frac{p}{q})\hspace{1mm} \kappa_b,
\]	
where $h^\vee$ is the dual Coxeter number, $p$ and $q$ are coprime integers, 
and $\kappa_b$ is the basic level. Then $W_{\fg, \kappa_{\circ}}$ has simple 
reflections given by the simple reflections 
of $W_f$ and the additional reflection
\begin{equation} \label{e:cases} s_{0, \kappa_\circ} = \begin{cases} 
s_{\check{\theta}_s} t^{-q 
	\check{\theta}_s }, & \text{if } (q, r) = 1 \\[2mm] 
s_{\check{\theta}_l} 
t^{-q/r 
	\check{\theta}_l }, & \text{if } (q, r) = r \end{cases}.
\end{equation}
\end{lemma}

\begin{proof} Recall the standard enumeration of the affine real coroots as 
	\[\check{\Phi} 
	\simeq \check{\Phi}_f 
	\times \mathbb{Z},\] where $\check{\Phi}_f$ denotes the finite  
	coroots. With 
	this 
	enumeration, the 
	element $\halpha_n \in \ft \oplus \mathbb{C} \mathbf{1}$, for $\halpha \in 
	\check{\Phi}_f$ 
	and $n 
	\in \mathbb{Z}$, is given by 
	\[ \halpha + n \frac{\kappa_{\circ}(\halpha, \halpha)}{2} \hspace{1mm} 
	\mathbf{1}.
	\]
	In particular $\halpha_n$ belongs to $W_{\fg, \kappa_\circ}$ if and only if 
	\[ \langle \halpha_n, 0 \rangle = n\frac{\kappa_\circ( \halpha, 
		\halpha)}{2}\]
	is an integer, which straightforwardly implies the claims of the lemma. 
\end{proof}

\subsubsection{}\label{sss:good-defin}
Having explicitly identified the Coxeter generators of the integral Weyl group, 
we will now obtain for 
most levels highest weights with prescribed stabilizers within it. Let us 
formulate this problem precisely. If we write $\Lambda_{G_s}$ for the weight 
lattice, and 
recall that $I_s$ denotes the Iwahori subgroup of $\fL G_s$, then
\[\widehat{\fg}_{\kappa_{\circ}}\mod^{I_s}\]
is compactly generated by the Verma modules $M_\lambda$, for $\lambda \in 
\Lambda_{G_s}$. 
In particular, this category has highest weights consisting of the weight 
lattice. Let us say a level $\kappa_{\circ}$ is {\em good} if for any finite
parabolic subgroup $W_\circ$ of $W_{\fg, \kappa_\circ}$ there exists 
an integral 
level $\kappa'$ and a simple Verma module 
\[ M_\nu \in \widehat{\fg}_{\kappa_{\circ} + \kappa'}\mod^{I_s}
\] 
whose highest weight has stabilizer $W_\circ$. Let us classify the good 
levels. 

\begin{pro}\label{p:adapt} Every irrational level is good. A rational level
	\[ \kappa_\circ = (-h^\vee + 
	\frac{p}{q}\hspace{1mm})\hspace{1mm}\kappa_b, 
	\]
	where $p$ and $q$ are coprime integers, is good if and only if $q$ is 
	coprime to the number $n(\fg)$ associated to 
	$\fg$ in Figure \ref{f:badprimes}. 
\end{pro}
\begin{center}
	\begin{figure}

		\begin{tabular}{ |m{2cm}  | m{2cm} | } 
			
			\hline
			\vspace{2mm}
			$\hspace{.9cm} \fg$ \vspace{2mm} & \vspace{2mm} 
			$\hspace{.8cm}n(\fg)$\vspace{2mm} 
			\\ 
			\hline
			
			\vspace{2mm} $\hspace{.9cm}A_n$ 	& \vspace{2mm} 
			$\hspace{.9cm}1$  
			\\[2mm] 
			$\hspace{.9cm}B_n$ 	& $\hspace{.9cm}2$  \\[2mm] 
			$\hspace{.9cm}C_n $	& $\hspace{.9cm}2$  \\[2mm] 
			$\hspace{.9cm}D_n $	& $\hspace{.9cm}2$ \\[2mm] 
			$\hspace{.9cm}E_6 $	& $\hspace{.7cm}2 \cdot 3$ \\[2mm]
			$\hspace{.9cm}E_7 $	& $\hspace{.7cm}2 \cdot 3$ \\[2mm]
			$\hspace{.9cm}E_8$  & $\hspace{.5cm}2 \cdot 3 \cdot 5$ \\[2mm]
			$\hspace{.9cm}F_4$ 	& $\hspace{.7cm}2 \cdot 3$ \\[2mm] 
			$\hspace{.9cm}G_2 $	& $\hspace{.7cm}2 \cdot 3 $\\[2mm]
			\hline
			
		\end{tabular}
		\caption{Bad primes for each simple Lie algebra}
		
		\label{f:badprimes}
	\end{figure}
\end{center}

\begin{proof} For $\kappa_{\circ}$ irrational, the claim is clear as $W_{\fg, 
		\kappa_{\circ}} \simeq W_f$. For $\kappa_{\circ}$ rational, which we 
	may take 
	to be negative, a weight $\lambda \in \Lambda_{G_s}$ is antidominant if and 
	only if 
	\[ \langle \lambda + \rho, \halpha_i \rangle \leqslant 0, 
	\quad 
	\text{for } i \in \mathscr{I}, \quad \quad  \text{and } \quad \quad 
	\begin{cases} 
	\langle  \lambda + \rho, \check{\theta}_s \rangle \geqslant -p , & \text{if 
	} 
	(q,r) = 1, \\ \langle  \lambda + \rho, \check{\theta}_l 
	\rangle \geqslant -p, & \text{if } (q,r) = r, \end{cases}
	\]
	as follows from \eqref{e:cases}. Let us write $\omega_i,$ for $i \in 
	\mathscr{I}$, for the 
	fundamental 
	weights, and write the dominant coroots as sums of simple coroots
	\[ \check{\theta}_s = \sum_{i \in \mathscr{I}} n_i \halpha_i, \quad \quad 
	\check{\theta}_l 
	= \sum_{i \in \mathscr{I}} m_i \halpha_i, \quad \quad \text{for } n_i, m_i 
	\in 
	\mathbb{Z}^{\geqslant 
		0}.
	\] 	
	Recall the standard correspondence between finite parabolic subgroups 
	$W_\circ$ of $W_{\fg, \kappa_{\circ}}$ and nonempty faces of the above 
	alcove, which 
	associates to a face the stabilizer of any interior point. It follows that, 
	after the transformation $\lambda \mapsto -\lambda - 
	\rho$, we 
	are looking for points of $\Lambda_{G_s}$ within the alcove with 
	vertices at
	\[ 0  \quad  \text{ and }  \quad \begin{cases} \frac{p}{n_i} \hspace{1mm}
	\omega_i, \quad \text{for } i \in \mathscr{I}, 
	& \text{if } (q,r) = 
	1, \\[2mm] \frac{p}{m_i} \hspace{1mm} \omega_i,  \quad \text{for } i \in 
	\mathscr{I}, 
	&\text{if } (q,r) 
	= r. 
	\end{cases}
	\]
	Recalling that we are free to replace $p$ by any element of $p + q 
	\mathbb{Z}$, 
	it is straightforward to see that we can find points of $\Lambda_{G_s}$ in 
	the 
	interior of every face of the alcove if and only if for each $i \in 
	\mathscr{I}$ one has%
	\[ \begin{cases} p \in (n_i, q), \quad \text{for } i \in \mathscr{I}, & 
	\text{if } (q,r) = 1, \\ p \in (m_i, q),\quad \text{for } i \in 
	\mathscr{I},  
	& 
	\text{if } (q,r) = r. \end{cases} 
	\]
	To see this, note these conditions are tautologically equivalent to being 
	able to realize each vertex of the alcove as a point of $\Lambda_{G_s}$, 
	hence they are necessary. To see they are sufficient, suppose they are 
	satisfied. Via this assumption, for any positive integer $N$ we may replace 
	$p$ with an element 
	of $p + q \mathbb{Z}$ so that 
	\[
	\frac{p}{n_i} \in \mathbb{Z}^{\geqslant N}, \quad \text{for all } i \in 
	\mathscr{I}.
	\]
	In particular, we may assume that $N$ is greater than the number of 
	vertices of the alcove, i.e. $N > \lvert I \rvert + 1$. In this case, 
	every face of the alcove contains an interior point expressible as a convex 
	combination of the vertices with coefficients in $\frac{1}{N}\mathbb{Z}$. 
	As such a convex combination is a point of $\Lambda_{G_s}$, we are done.

	Finally, recalling the $n_i$ and $m_i$ for each type, cf. Plates I-IX of 
	\cite{bour}, yields the entries of Figure \ref{f:badprimes}. Namely, 
	if $n_i$ is prime, $p \in (n_i, q)$ if and only if $(n_i, q) = 1$, and if 
	$n_i$ 	is composite, then each of its prime divisors occurs as another 
	$n_{i'}$.  
	\end{proof} 
	
\subsubsection{}

In what follows, we will be most concerned with the Whittaker category
on $\fL G/I$, i.e., with $\Whit_{\lambda}^{\aff}$ for $\lambda = 0$.
In this case, we replace the notation $\lambda$ by the level $\kappa$.

In other words, $\Whit_{\kappa}^{\aff} \coloneqq \Whit_0^{\aff}$.
We similarly denote the summands $\Whit_{0,y}^{\aff}$ 
of this category considered in
Theorem \ref{t:whit-blocks} by $\Whit_{\kappa,y}^{\aff}$.

\subsubsection{}

Let us obtain, for good levels $\kappa_\circ$, the Kac--Moody realization of 
blocks of the Whittaker category. Fix a double coset, whose minimal length 
element we denote by $y$,  in
\[ W_{\fg, \kappa_{\circ}} \backslash \widetilde{W} / W_f.
\] 

\begin{cor}\label{c:wtokm} If $\kappa_{\circ}$ is good, then for any $y$ as 
	above the 
	corresponding  block $\Whit_{\kappa, y}^{\aff}$ admits an
	equivalence with a block of $\widehat{\fg}_{\kappa_{\circ} + 
		\kappa'}\mod^{I, y^{-1} \cdot 0},$
	for some integral level $\kappa'$. 
	
\end{cor}

\begin{proof} By Proposition \ref{p:neutr} and Theorem \ref{t:whit-blocks}, it 
is 	enough to produce a simple Verma module in 
	\[\widehat{\fg}_{\kappa_{\circ} + 
		\kappa'}\mod^{I, y^{-1} \cdot 0}\]
	whose (antidominant) highest weight $\Theta$ has stabilizer under the dot 
	action of $W$ 
	given 
	by 
	\[y^{-1} W_{\fg, \kappa_{\circ}} y \hspace{1mm} \cap \hspace{1mm} W_f.	\]
	By \eqref{e:posy}, it is equivalent to produce a simple Verma module 
	in 
	\[\widehat{\fg}_{\kappa_{\circ} + 
		\kappa'}\mod^{I} \]
	whose highest weight $y \cdot \Theta$ has stabilizer given by 
	\begin{equation} \label{e:stab} W_{\fg, \kappa_{\circ}} \cap \hspace{1mm} 
	yW_fy^{-1}. \end{equation}
	The latter is provided by Proposition \ref{p:adapt}, as desired. 
\end{proof}

\begin{re}\label{r:sf-2}

As in Remark \ref{r:sf-1}, and at the risk of 
redundancy, we emphasize: 
by construction, the integral Weyl group of the above block
is identified as a Coxeter system with $W_{\fg,\kappa}$, and the stabilizer of 
$\Theta$ in 
$W_{\fg,\kappa}$ is 
$W_{\fg, \kappa_{\circ}}  \cap yW_fy^{-1}$.

\end{re}

\begin{re} We recall that our running assumption in this section is that 
$\kappa$ is negative, as will be important for the arguments when we get to the 
fundamental local equivalences. However,  everything until this point, in 
particular our analysis of the twisted Whittaker categories, applies 
to an arbitrary level $\kappa$. 
\end{re}

\subsection{From $\widehat{\fg}$-modules to $\widehat{\check{\fg}}$-modules}

\subsubsection{}
To relate blocks of Category $\mathscr{O}$ for $\gmk$ and $\cgk$, we would like 
to use the following result of Fiebig.

Let $\mathfrak{k}$ be an affine Lie algebra with Cartan subalgebra $\fh$. Fix  
$\alpha \in \fh^*$ such that the Verma module $M_\alpha$ is 
simple, and write $\OO_\alpha$ for the corresponding block of Category 
$\mathscr{O}$ for $\mathfrak{k}$. Let us write $W_\alpha$ for the integral Weyl 
group of $\alpha$, and $W_\alpha^\circ$ for its subgroup stabilizing $\alpha$ 
under the dot action.

\begin{theo}[\cite{fiebig} Thm. 4.1]\label{t:fiebig} 
As an abelian highest weight 
category, $\OO_\alpha$ is determined up to equivalence by the Coxeter system 
$W_\alpha$ along with its subgroup $W_\alpha^\circ$. \end{theo}

\begin{re} Theorem \ref{t:fiebig}, as written in {\em loc. cit.}, applies to 
symmetrizable Kac--Moody algebras, and in particular affine Kac--Moody algebras. 
Recall the latter consists of a Laurent polynomial version of the affine Lie 
algebra along with a degree operator $t \partial_t$. One may apply it to the 
present situation as follows. At any 
noncritical level, Category $\mathscr{O}$ for the affine Lie algebra 
canonically embeds as a Serre subcategory of Category $\mathscr{O}$ for the 
affine Kac--Moody algebra. Namely, one sets $t \partial_t$ to act by the 
semisimple part of $-L_0$, where $L_0$ is the Segal--Sugawara energy operator.  
\end{re}

\subsubsection{}
To apply Theorem \ref{t:fiebig} in our situation, we must relate $W_{-\kappa, 
\fg}$ and 
$W_{\check{\kappa}, 
\check{\fg}}$.

To do so, fix a level $\kappa_{\circ}$ for $\fg$. Let us write 
$\check{Q}$ for the coroot lattice and $Q$ for the 
root lattice of $\fg$. Associated to $\kappa_\circ$ is a map 
\[ (\kappa_\circ - {\kappa_{\fg,c}}): \check{Q} \rightarrow Q 
\otimes_{\mathbb{Z}} 
\mathbb{C},
\]
where $\kappa_{\fg, c}$ denotes the critical level for $\fg$. In particular we 
may 
consider 
the 
sublattice of $\check{Q}$ given by 
\[
\check{Q}_{\kappa_\circ} := \{ \check{\lambda} \in \check{Q}: (\kappa_\circ - 
\kappa_{\fg, c}) 
(\check{\lambda}) \in Q \}.
\]
Suppose that $\kappa_\circ$ is noncritical. Recall that, if we write 
$\check{\kappa}_{\check{\fg}, c}$ for the critical level for $\check{\fg}$, 
then by 
definition, the dual nondegenerate bilinear form to $\kappa_{\circ} - 
\kappa_{\fg, c}$ 
is $\check{\kappa}_\circ - \check{\kappa}_{\check{\fg},c}$. It follows we have 
a canonical identification
\begin{equation} \label{e:coxeq3} (\kappa_\circ - \kappa_{\fg, c}):  
\check{Q}_{\kappa_\circ} 
\simeq Q_{\check{\kappa}_\circ}: 
(\check{\kappa}_\circ - \check{\kappa}_{\check{\fg}, c}).
\end{equation}

\subsubsection{}
With this, we may canonically identify the integral Weyl groups on the opposite 
sides of quantum Langlands duality. 

\begin{theo}\label{t:intw} For any level $\kappa_\circ$, under the 
identification $W \simeq W_f \ltimes \check{Q}$ one has
\begin{equation} \label{e:subafw}
W_{\fg, \kappa_\circ} \simeq W_f \ltimes \check{Q}_{\kappa_\circ},
\end{equation}
	 Moreover, for a noncritical 
	level $\kappa_\circ$, there is a canonical isomorphism of Coxeter systems
	\begin{equation}\label{e:coxeq2} W_{\fg, \kappa_\circ} \simeq 
	W_{\check{\fg}, 
		\check{\kappa}_\circ}.
	\end{equation}
\end{theo}

\begin{re}
As $G$ and $\check{G}$ in general have different 
affine Weyl groups, there is, perhaps, something surprising about 
Theorem \ref{t:intw}. 
\end{re}

\begin{proof}[Proof of Theorem \ref{t:intw}] We begin with \eqref{e:subafw}. Recall from the proof of 
Lemma \ref{l:simprefs} that for a finite coroot $\halpha$ and integer 
$n$, the affine coroot $\halpha_n$, belongs to $W_{\fg, \kappa_\circ}$ if and 
only if
	\[ \langle \halpha_n, 0 \rangle = n\frac{\kappa_\circ( \halpha, 
		\halpha)}{2}.\]
	is an integer. This integrality condition may be rewritten as
	\[ n \kappa_\circ(\halpha) \in \mathbb{Z} \alpha, \]
	which in turn is equivalent to  $n \halpha \in \check{Q}_{\kappa_\circ}.$ 
    As the affine 
	reflection in $W$ corresponding to $\halpha_n$ 
	is explicitly given by $
	s_{\halpha} t^{n \halpha},$
	it follows that we have an inclusion 
	\begin{equation} \label{e:coxinc} W_{\fg, \kappa_\circ} \subset W_f \ltimes 
	\check{Q}_{\kappa_\circ},
	\end{equation}
	and that the left-hand side includes the translations $t^{n \halpha},$ for 
	$n \halpha$ as above. To see that \eqref{e:coxinc} is an equality, it 
	suffices to see that $\check{Q}_{\kappa_\circ}$ lies in the left-hand side. 
	But if we write an element $\check{\lambda}$ of $\check{Q}$ as a linear 
	combination
	\[ \check{\lambda} = \sum_{i \in \mathscr{I}} n_i \halpha_i, \quad \quad 
	\text{for } 
	i \in \mathscr{I},
	\]
	we have that $\kappa_\circ(\lambda)$ lies in $Q$ if and only if 
	$\kappa_\circ( n_i \halpha_i)$ lies in $Q$, for all $i \in \mathscr{I}$. 
	Considering 
	the affine coroots $(\halpha_i)_{n_i}$ for $i \in \mathscr{I}$ and 
	composing the
	translational parts of their reflections yields the desired equality.

	Let us use the equality \eqref{e:subafw} to prove \eqref{e:coxeq2}. Namely, 
	via \eqref{e:coxeq1}, we may assume that $\kappa_\circ$ is positive. Under 
	this assumption, we will show the composite identification 
	\[ W_{\fg, \kappa_\circ} \overset{\eqref{e:subafw}}{\simeq} W_f \ltimes 
	\check{Q}_{\kappa_\circ} \overset{\eqref{e:coxeq3}}{\simeq} 
	W_f \ltimes Q_{\check{\kappa}_\circ} \overset{\eqref{e:subafw}}{\simeq}  
	W_{\check{\fg}, 
		\check{\kappa}_\circ}
	\]
	is an isomorphism of Coxeter systems. That is, we claim the sets of simple 
	reflections are exchanged under the identification 
	\begin{equation} \label{e:coxla} (\kappa_{\circ} - \kappa_{\fg, c}): W_f 
	\ltimes 
	\check{Q}_{\kappa_\circ} {\simeq} \hspace{1mm}
	W_f \ltimes Q_{\check{\kappa}_\circ}.
	\end{equation}
	If $\kappa_\circ$ is irrational, this is clear, as both sides are $W_f$. 
	Otherwise, let us write the level as
	\begin{equation} \label{e:frac} \kappa_\circ = (-h^\vee + 
	\frac{p}{q}\hspace{1mm})\hspace{1mm}\kappa_b, 
	\end{equation}
where $p, q$ are positive coprime integers. Recall the affine simple reflection 
$s_{0, \kappa_{\circ}}$ from Lemma \ref{l:simprefs}. Applying \eqref{e:coxla} 
to it, we obtain, 
	writing $\theta_s$ for the short dominant root and $\theta_l$ for the long 
	dominant root, 
	\begin{equation} \label{e:cases2} (\kappa_\circ - \kappa_{\check{\fg}, 
		c})\hspace{1mm} s_{0, \kappa_\circ} = \begin{cases} 
	s_{\theta_l} t^{-p 
		\theta_l }, & \text{if } (q, r) = 1 \\[2mm] s_{\theta_s} 
	t^{-p 
		\theta_s }, & \text{if } (q, r) = r \end{cases},
	\end{equation}
	To finish, recall that $\check{\kappa}_\circ$ is 
	given by 
	\[ \check{\kappa}_\circ = (-h^\vee_{\check{\fg}} + \frac{q}{pr} 
	\hspace{1mm} ) \hspace{1mm} \check{\kappa}_{\check{\fg},b}, 
	\]
	where $h^\vee_{\check{\fg}}$ and $\check{\kappa}_{\check{\fg},b}$ are the 
	dual Coxeter number and basic level for $\check{\fg}$, respectively. 
	Comparing the analog of \eqref{e:cases} for $\check{\fg}$ to 
	\eqref{e:cases2} shows the intertwining of the affine simple reflections by 
	\eqref{e:coxla}, as desired. \end{proof}

\subsubsection{}
We may apply these as follows. Suppose $\kappa$ and $\check{\kappa}$ are dual 
levels, and $\kappa'$ is an integral level for $G_s$. Suppose we are given 
a $y \in 
\widetilde{W}$ of minimal length in $W_{\fg, -\kappa + \kappa'}y$, and 
a  simple Verma 
module 
\[M_\mu \in \widehat{\fg}_{-\kappa + \kappa'}\msf{-mod}^{I_s, y^{-1} \cdot 
0}. 
\]
Suppose we are further given a simple Verma module $M_{\check{\nu}}$ in 
$\widehat{\check{\fg}}_{\check{\kappa}}\msf{-mod}^{\check{I}_s}$, such that the 
stabilizers of $\mu$ and $\check{\nu}$ are identified via
\begin{equation}
\label{e:stabs} y^{-1}W_{\fg, -\kappa + \kappa'}y \simeq W_{\fg, -\kappa + 
\kappa'} \overset{\eqref{e:coxeq1}}{\simeq} W_{\fg, \kappa} 
\overset{\eqref{e:coxeq2}}{\simeq} W_{\check{\fg}, 
\check{\kappa}}. 
\end{equation}
\begin{cor} In the above situation, there is a $t$-exact equivalence 
	\[ \widehat{\fg}_{-\kappa + \kappa'}\msf{-mod}^{I_s, y^{-1} \cdot 
		0}_\mu \simeq 
		\widehat{\check{\fg}}_{\check{\kappa}}\msf{-mod}^{\check{I}_s}_{\check{\nu}}
	\]
	\label{c:theend}
\end{cor}

\begin{proof} For either category, which we temporarily denote by $\scc$, and 
by $\scc^{\heartsuit, c}$ the finite length objects in its heart, the canonical 
map 
	\[ D^b( \scc^{\heartsuit, c}) \rightarrow \scc
	\]
realizes the latter as the ind-completion of the former. To see this, note that 
since the blocks contain a simple Verma module, $\scc^{\heartsuit, c}$ indeed 
consists of compact objects, and the fully faithfulness may be checked from 
Verma to dual Verma modules. Either $\scc^{\heartsuit, c}$ is a block of 
Category $\mathscr{O}$ for the corresponding affine algebra, whence 
we are done by the assumptions on $\mu$ and $\check{\nu}$, the identification 
of Coxeter systems
\eqref{e:stabs}, and 
Theorem \ref{t:fiebig}. 
\end{proof}

\subsection{The tamely ramified fundamental local equivalence}
Recall we have assumed that $G$ is simple of adjoint type and that $\kappa$ is negative. 
\begin{theo}\label{t:fle-tame} 

For $\kappa$ good, there is a $t$-exact equivalence
	\begin{equation} \label{e:tbs} 
	\Whit_{\kappa}^{\aff} 
	\overset{Sect. \ref{sss:indexing}}{\simeq} 
	D_{-\kappa}(I \backslash \fL G / 
	\fL N^-, \psi) \simeq 
	\cgk\mod^{\check{I}} 
	\end{equation}
	\end{theo}
\begin{proof} Recall 
	our identification from \ref{sss:iwd} of the 
	isomorphism classes of simple objects in the left-hand side of 
	\eqref{e:tbs} with the coset 
	space
	\[ \widetilde{W} / W_f. \]
	We showed in Lemma \ref{l:dcs} that each orbit of $W_{\fg, -\kappa}$ on the 
	latter contains a minimal 
	element $y$ with respect to the Bruhat order. In Corollary \ref{c:wtokm} we 
	showed
	the corresponding block $\Whit_{\kappa, y}^{\aff}$ of the left-hand side 
	of \eqref{e:tbs} identifies
	with a block of twisted $I_s$-integrable modules for 
	$\widehat{\fg}_{-\kappa + 
	\kappa_\circ}$, for an integral level $\kappa'$. 
	As discussed in the proof of Corollary \ref{c:wtokm}, its integral Weyl 
	group 
	identifies as a Coxeter system with 
	$W_{\fg, -\kappa}$, and with this identification the stabilizer of the 
	highest 
	weight of a simple Verma module is given by 
	\begin{equation} \label{e:dada} W_{\fg, -\kappa} \cap \hspace{1mm} y W_f 
	y^{-1}.\end{equation}

	On the other hand, if we denote by $\check{\Lambda}_G$ the cocharacter 
	lattice of 
	$T$, 
	i.e. the character lattice of $\check{T}$, recall that $\widetilde{W}$ is 
	explicitly 
	\[ \widetilde{W} \simeq W_f \ltimes \check{\Lambda}_G,
	\]
	Consider its action on $\check{\Lambda}_G$, where $W_f$ acts by the dot 
	action 
	and an element $\check{\lambda}$ in $\check{\Lambda}_G$ acts as translation 
	by 
	{\em minus} $\check{\lambda}$. Acting on $-\check{\rho}$ yields a 
	$\widetilde{W}$-equivariant identification 
	\begin{equation} \label{e:sgn} \widetilde{W}/W_f \simeq \check{\Lambda}_G.
	\end{equation}

	With this, recalling that $\kappa$ is negative, the restriction of 
	\eqref{e:sgn} along the composition
	\[
	W_{\check{\fg}, \check{\kappa}} \overset{\eqref{e:coxeq1}}{\simeq} 
	W_{\check{\fg}, -\check{\kappa}} \overset{\eqref{e:coxeq2}}{\simeq} W_{\fg, 
		-\kappa} \hookrightarrow \widetilde{W} 
	\]
	yields the standard dot action of $W_{\check{\fg}, \check{\kappa}}$ on 
	$\check{\Lambda}_G$. Moreover, under the equivalence \eqref{e:sgn}, for any 
	reflection $s_\alpha$ in $W_{\check{\fg}, \check{\kappa}}$ and element $x$ 
	of 
	$\widetilde{W}$, one has that
	\begin{equation} \label{e:bru} xW_f \hspace{1mm}\leqslant 
	\hspace{1mm}s_\alpha 
	x W_f \quad\quad \text{ if 
		and 
		only if } \quad\quad  x 
	\cdot-\check{\rho}\hspace{1mm} 
	\leqslant\hspace{1mm} s_\alpha x \cdot - \check{\rho}, \end{equation}
	where the latter $\leqslant$ denotes the standard partial order on 
	$\check{\Lambda}_G$. To see this note that, if we write $\Phi^+$ for the 
	positive real coroots of $\check{\fg}$ and $\Phi_f$ for the finite coroots 
	of 
	$\check{\fg}$, both sides of \eqref{e:bru} are equivalent to 
	\[ x^{-1}(\alpha) \in \Phi^+ \cup \Phi_f.
	\]

	Using \eqref{e:bru}, we may describe the block decomposition of 
	\[ \cgk\mod^{\check{I}}
	\]
	as follows. In its usual formulation, due to Deodhar--Gabber--Kac 
	\cite{dgk}, blocks correspond to $W_{\check{\fg}, \check{\kappa}}$ dot 
	orbits on $\check{\Lambda}_G$, and each contains a unique simple Verma 
	module. Under the identification of its highest weights with 
	$\widetilde{W}/W_f$ via \eqref{e:sgn}, each orbit of $W_{\check{\fg}, 
		\check{\kappa}}$ contains a minimal element $y$ with respect to the 
		Bruhat 
	order. By \eqref{e:bru}, the corresponding Verma module is antidominant, 
	i.e. simple, and has stabilizer 
	\begin{equation} \label{e:tada} W_{\check{\fg}, \check{\kappa}} \cap 
	\hspace{1mm} y W_f 
	y^{-1}.
	\end{equation}
	Comparing Equations \eqref{e:dada} and \eqref{e:tada}, we are done by 
	Corollary \ref{c:theend}. 
\end{proof}

\subsubsection{}
To conclude, we record some properties of the above equivalence, which will be 
used in Proposition \ref{p:isog} below.

\begin{pro}\label{p:whogoeswhere} For $\kappa$ good, the obtained 
	equivalence
	\[ \mathsf{F}: D_\kappa( \fL N, \psi \backslash \fL G / 
	I) \simeq 
	\cgk\mod^{\check{I}}
	\]
	interchanges, for any $\check{\lambda}$ in $\check{\Lambda}_G$, the 
	(co)standard and simple objects, i.e.,
	\[
	\mathsf{F}( \hspace{.5mm}j_{\check{\lambda}, !}^\psi\hspace{.1mm} )\simeq 
	M_{\check{\lambda}}, \quad 
	\quad 
	\mathsf{F}(\hspace{.5mm} j_{\check{\lambda}, !*}^\psi \hspace{.1mm})\simeq 
	L_{\check{\lambda}}, \quad 
	\quad \text{and } \quad \quad 
	\mathsf{F}(\hspace{.5mm}j_{\check{\lambda}, *}^\psi\hspace{.1mm}) \simeq 
	A_{\check{\lambda}}. 
	\]
\end{pro}
Proposition \ref{p:whogoeswhere} immediately 
implies, via the constructions of the appendix, a similar statement for 
general reductive groups $G$ at negative level. 
 
 \begin{re} \label{r:Soergelmethods}Having obtained the tamely ramified 
 fundamental local equivalence 
 for 
 good levels, let us outline a variant of the proof which may be 
 desirable.
 
 Presently, we relate blocks of $\Whit_\kappa^{\aff}$ and 
 $\cgk\mod^{\check{I}}$ by relating the former to $\gk\mod^{I}$ and applying 
 Fiebig's results. However, it should be possible to adapt the arguments of 
 Bezrukavnikov--Yun on $\mathbb{V}$-functors provided in Section 4 and 5 of 
 \cite{by} 
 to $\Whit_{\kappa}^{\aff}$, and thereby identify each block with a category of 
 (possibly singular) Soergel modules. Comparing this with the similar 
 identification provided by 
 Fiebig, and matching the combinatorics exactly as in the proof of Theorem 
 \ref{t:fle-tame} should yield the desired equivalence. 
 
 This would remove the assumption of 
 goodness on $\kappa$, and such a description of the Whittaker 
 category should be equally applicable in other sheaf-theoretic contexts, e.g. 
 metaplectic Whittaker sheaves over function 
 fields.

 \end{re}

\subsection{Parahoric fundamental local equivalences}\label{ss:parahoric}

\subsubsection{}Recall the canonical bijection between the simple roots of 
$\fg$ and 
$\check{\fg}$, which were indexed by $\mathscr{I}$. In particular, for a 
standard parahoric subgroup 
\[ P \subset \fL G,\]
which  corresponds to a subset $\sJ$ of $\mathscr{I}$, we may 
associate a dual parahoric 
\[\check{P} \subset \fL\cG. \]
\subsubsection{}
Let us obtain 
a parahoric extension of the tamely ramified fundamental local equivalence. In 
particular, we continue to assume that $\kappa$ is negative and $G$ is simple 
of adjoint type. 

\begin{theo}\label{t:galu} For $\kappa$ good, there is an 
	equivalence
	\begin{equation}  D_{\kappa}( \fL N, \psi \backslash 
	\fL G / P) \simeq 
	\cgk\mod^{\check{P}}.
	\end{equation}
\end{theo}

\begin{proof}It is enough to produce an equivalence, which we 
	denote provisionally by a dotted line, fitting into commutative diagram
	\begin{equation} \label{e:diag} \begin{gathered}
	\xymatrix{ \hspace{-7mm} D_{-\kappa}(P \backslash \fL G / 
		\fL N^-, 
		\psi)^{\heartsuit, c}  
		\ar@{--}[rr]^\sim \ar[d]_{\pi^{!*}} 
		& &
		\cgk\mod^{\check{P}, \heartsuit, c} \ar[d]^{\on{Oblv}} 
		\\\hspace{-7mm} 
		D_{-\kappa}(I \backslash \fL G / \fL N^-, \psi )^{\heartsuit, 
			c} 
		\ar@{-}[rr]^{\eqref{e:tbs}} && \cgk\mod^{\check{I}, 
			\heartsuit, c}, }\end{gathered}
	\end{equation}
	where we normalize the pull-back $\pi^{!*}$ associated to $\pi: \fL G / I
	\rightarrow \fL 
	G/ P$ to be $t$-exact. Here, as in the proof of Corollary \ref{c:theend}, 
	the superscripts $\heartsuit, c$ denote compact objects in the heart, which 
	in the present situation is equivalent to finite length objects in the 
	heart. Noting 
	that both vertical arrows in \eqref{e:diag} 
	are full 
	embeddings 
	of Serre subcategories,  it is enough 
	to 
	see that the essential images of the simple objects under are intertwined 
	by \eqref{e:tbs}. To see this claim, 
	recall 
	we denoted the subset of simple roots corresponding to the dual parahorics  
	$P$ and $\check{P}$ by $\sJ$. With this, the simple objects on the 
	Whittaker 
	side lying in the essential image are 
	intermediate extensions from the orbits
	\[ Ix\fL N^-, \quad \quad \text{for } x \in \widetilde{W},
	\]
	where $x$ satisfies the three conditions 
	\begin{align} \label{e:whithw}
	s_j x < x, \hspace{2mm}  \forall j \in \sJ, \quad \quad xs_i < x, 
	\hspace{2mm}  
	\forall i 
	\in \mathscr{I}, \quad \quad \text{and } \quad \quad 	
	W_\sJ x s_i < W_\sJ x,  \hspace{2mm}  \forall i \in \mathscr{I}.
	\end{align}
	The simple objects on the Kac--Moody side lying in the essential image have 
	highest weights $\check{\lambda}$ satisfying 
	\begin{equation} \label{e:parhw} s_j \cdot \check{\lambda} < 
	\check{\lambda}, 
	\quad \forall j \in \sJ.
	\end{equation}
	Write $\check{\lambda} = x \cdot -\check{\rho}$, for $x \in \widetilde{W}$ 
	acting as in \eqref{e:sgn}. We may assume $x$ is of maximal length in its 
	left 
	$W_f$ coset, in which case \eqref{e:parhw} is equivalent to $x$ satisfying 
	the 
	three conditions 
	\begin{equation}\label{e:parhw2}
	s_j x < x, \hspace{2mm}  \forall j \in \sJ, \quad \quad xs_i < x, 
	\hspace{2mm}  
	\forall i 
	\in \mathscr{I}, \quad \quad \text{and } \quad \quad 	
	s_j x W_f < x W_f,  \hspace{2mm}  \forall j \in \mathscr{J}.
	\end{equation}
	We finish by noting that \eqref{e:whithw} and \eqref{e:parhw2} describe the 
	same subset of $\widetilde{W}$, namely the unique elements of maximal 
	length in 
	double cosets $W_\sJ x W_f$, such that the double coset satisfies
	\[  x( \check{\Phi}_f) \cap \check{\Phi}_{\sJ} = \emptyset,
	\]
	where $\check{\Phi}_f$ denotes the finite coroots and $\check{\Phi}_{\sJ}$ 
	the 
	coroots of the Levi associated to $\mathscr{J}$. \end{proof}

\subsubsection{}

We next observe that the analog of Proposition \ref{p:whogoeswhere} again holds 
in the parahoric setting. As notation for the highest weights of 
$\cgk\mod^{\check{P}}$, let us write
\[ \check{\Lambda}_{G, P}^+ = \{ \check{\lambda} \in 
\check{\Lambda}_G:\hspace{2mm} \langle 
\alpha_j, \check{\lambda} \rangle \in \mathbb{Z}^{\geqslant 0}, \text{  for 
	all 
} j \in 
\sJ \}. 
\]

\begin{pro} For $\kappa$ good, the obtained equivalence 
	\[ \mathsf{F}: D_\kappa( \fL N, \psi \backslash \fL G / 
	P) \simeq 
	\cgk\mod^{\check{P}} \]
	interchanges, for $\check{\lambda}$ in $\check{\Lambda}^+_{G, P}$, the 
	corresponding (co)standard and simple objects.
	
\end{pro}

\begin{proof} For simple objects this is contained in the proof of Theorem 
	\ref{t:galu}. For the remaining claims, note that the standard and 
	costandard 
	objects of the parahoric categories may be obtained from the standard and 
	costandard objects of the Iwahori categories by appling the left and right 
	adjoints of the vertical arrows in \eqref{e:diag}, respectively. Hence the 
	claim follows from Proposition \ref{p:whogoeswhere}.
\end{proof}
Via the reductions in the appendix, a similar statement applies for 
general 
reductive $G$ at negative level. 

\subsubsection{}
We finish with two remarks.

\begin{re} Applying Theorem \ref{t:galu} in the maximal 
case, i.e., for the parahorics given by the arc groups $\fL^+ G$ and $\fL^+ 
\cG$, we obtain the spherical fundamental local equivalence for good 
levels, namely
\[ \Whit_{\kappa}^{\on{sph}} \coloneqq
D_\kappa( \fL N, \psi \backslash \fL G  / \fL ^+ G) \simeq 
\cgk\mod^{\fL^+ \cG}.\]
\end{re}

\begin{re} We would like to record here the expectation that, for dual 
parahorics 
$P$ and $\check{P}$ as above, local quantum Langlands duality exchanges the 
operations of taking $P$ and $\check{P}$ invariants. For the Iwahori and arc 
subgroups, this already appears in the literature \cite{campja}, 
\cite{dennispro}. However, while one has a canonical bijection between the 
affine simple roots 
for $\fg$ and $\check{\fg}$, we do {\em not} expect an interchanging of the 
corresponding invariants of more general parahoric subgroups. Indeed, already 
the analog of Theorem \ref{t:galu} will typically fail, unless $\check{\kappa}$ 
is 
integral and $\fg$ is simply laced. 
\end{re}

\appendix

\section{Proof in the general case} \label{a:reduc}

In this appendix, we spell out how to deduce the general case of the 
conjectures from the case of $G$ of adjoint type and $\kappa$ a negative 
level. While we write the reductions in the tamely ramified case, they apply 
{\em mutatis mutandis} in the parahoric cases as well.

\subsection{Good levels for general $G$} Recall the notion of a good level for 
a simple group, cf. Section \ref{sss:good-defin} and Proposition \ref{p:adapt}. 
Let us a say a level $\kappa$ for a reductive group $G$ is good if it is good 
after restriction to each simple factor of $G$.

The following reductions in 
fact show the general case of the conjectures for reductive $G$ and $\kappa$ 
good reduce to the case of simple $G$ of adjoint type and $\kappa$ a good 
negative 
level. In particular, via Theorems \ref{t:fle-tame} and \ref{t:galu}, we obtain 
the fundamental local equivalences for general $G$ at good levels.

\subsection{Finite isogenies} \label{ss:isog} Suppose we are given a finite 
isogeny of 
pinned connected reductive groups \[\iota: G_1 \rightarrow G_2.\]The morphism 
$i$ 
yields a closed embedding of 
affine flag varieties, and hence 
a fully faithful embedding
\begin{equation}
\label{e:ewhit}D_\kappa( \mathfrak{L}N_1, \psi \backslash \fL G_1 / 
I_1) 
\rightarrow D_\kappa(\mathfrak{L}N_2, \psi \backslash \fL G_2 / I_2).
\end{equation}
Consider the Langlands dual isogeny of connected reductive groups 
\[\check{\iota}: 
\check{G}_2 \rightarrow 
\check{G}_1.\] Associated to $\check{\iota}$ is a fully faithful restriction 
map 
\begin{equation} \label{e:ecato}
\cgk\mod^{\check{I}_1} \rightarrow \cgk\mod^{\check{I}_2}.
\end{equation}
To see the claimed fully faithfulness, one may use the following lemma. 

\begin{lem} Suppose one is a given a fibre sequence of quasi-compact affine 
	group schemes 
	\[ 1 \rightarrow K \rightarrow H \rightarrow Q \rightarrow 1\]
	where $K$ is of finite type, the prounipotent radical $H^u$ is of finite 
	codimension in $H$, and $pt/K$ is homologically 
	contractible, i.e.,   \[H^*(pt/K, \mathbb{Q}) \simeq \mathbb{Q}.\]Then 
	for any $D(Q)$-module $\scc$, the restriction map 
	$\scc^{Q} \rightarrow \scc^H
	$
	is fully faithful. 
	
\end{lem}

\begin{proof} For $\scc \simeq D(Q)$, this is exactly the assumption of 
	homological contractibility. The case of general $\scc$ follows from taking 
	its cobar resolution as a $D(Q)$-module and using the commutation of $Q$ 
	and $H$ invariants with colimits. 
\end{proof}

We may apply the lemma to $\check{I}_2 \rightarrow \check{I}_1$, as the kernel 
identifies with the kernel of $\check{\iota}$. Combining these assertions, we 
obtain:

\begin{pro}\label{p:isog} Suppose that one has an equivalence of the form 
	\eqref{e:tfle} for $(G_2, \kappa)$ and $(\check{G}_2, 
	\check{\kappa})$. 
	Assume it exchanges the full subcategory generated under shifts and 
	colimits by the 
	Whittaker 
	sheaves
	\begin{align*} &j_{\check{\lambda}, !}^\psi, \quad  \quad \text{for } 
	\check{\lambda} \in  
	\check{\Lambda}_{G_1} \intertext{
		with the full subcategory generated under shifts and colimits by the 
		Verma modules}  
	&M_{\check{\lambda}}, \quad \quad  \text{for }
	\check{\lambda} \in \check{\Lambda}_{G_1}.\end{align*}
	Then it induces an equivalence of 
	the form \eqref{e:tfle} for $(G_1, \kappa)$ and $(\check{G}_1, 
	\check{\kappa})$, fitting into a commutative diagram
	\[\xymatrix{ \hspace{-7mm} D_\kappa(\mathfrak{L}N, \psi \backslash 
		\fL G_1 / I_1) 
		\ar@{-}[rr]^\sim \ar[d]_{\eqref{e:ewhit}} & &
		\cgk\mod^{\check{I}_1} \ar[d]^{\eqref{e:ecato}} \\\hspace{-7mm} 
		D_\kappa(\mathfrak{L}N, 
		\psi \backslash 
		\fL G_2/ I_2) \ar@{-}[rr]^\sim & & \cgk\mod^{\check{I}_2}.  }
	\]

\end{pro}

\subsection{Products} By the preceding subsection, we may replace our group, 
after passing to a finite quotient thereof, by the product of a semisimple 
group of adjoint type and a torus. We next reduce to the case of a single 
factor.

Suppose that $G$ factors as a product of pinned connected reductive groups $G 
\simeq G_1 \times G_2.$
Associated to this is a tensor product decomposition  
\[ D_\kappa(\mathfrak{L}N, \psi \backslash\fL G/I) \simeq 
D_{\kappa_1}(\fL N_1, \psi_1 \backslash \fL G_1 / I_1) 
\otimes D_{\kappa_2}(\fL 
N_2, \psi_2 \backslash \fL G_2 / I_2).
\]
On the Langlands dual side, we obtain a decomposition $\cG \simeq \cG_1 \times 
\cG_2$, and a similar tensor product decomposition 
\[\cgk\mod^{\check{I}} \simeq \widehat{\check{\mathfrak{g}}}_{1, 
	\kappa_1}\hspace{-1mm}\mod^{\check{I}_1}
\otimes \hspace{1mm} \widehat{\check{\mathfrak{g}}}_{2, 
	\kappa_2}\hspace{-1mm}\mod^{\check{I}_2}.
\]
In particular, to provide an equivalence as in \eqref{e:tfle} for $G_1 
\times G_2$, it is enough to do so for $G_1$ and $G_2$ separately. 

\subsection{Tori}Let us dispose of the torus factor of $G$. Given  dual tori 
$T$ and $\check{T}$, it is clear that both sides of 
\eqref{e:tfle} canonically identify as 
\[ \bigoplus_{\check{\lambda} \in \check{\Lambda}_T} \mathsf{Vect},
\]
corresponding to the components of the affine Grassmannian of $T$ and the Fock 
modules 
for $\widehat{\check{\ft}}_{\check{\kappa}}$, respectively. 

\subsection{Positive level} We have reduced the conjecture to $G$ simple of 
adjoint type, and a $\kappa$ arbitrary. We now reduce to $\kappa$ negative via 
cohomological duality.

For a connected reductive group $G$, there is a canonical 
$D_\kappa(\fL G)$-equivariant 
duality 
\[ D_\kappa(\fL G / I)^\vee \simeq D_{-\kappa}(\fL G / I),
\]
induced by Verdier duality. This yields a duality of the 
Whittaker invariants, cf. Section 4 of \cite{lpw}, i.e., 
\begin{equation} \label{e:whitd}  D_\kappa(\fL N, \psi \backslash 
\fL G / 
I)^\vee \simeq D_{-\kappa}(\fL N, 
-\psi \backslash \fL G / I).
\end{equation}
On the Kac--Moody side, recall that semi-infinite 
cohomology (defined with respect to any compact open subalgebra) induces an 
$D_\kappa(\fL G)$-equivariant duality 
\[ \cgk\mod^\vee \simeq 
\widehat{\check{\mathfrak{g}}}_{-\check{\kappa}}\mod, \]
cf. Section 9 of \cite{mys}.Passing to $\check{I}$ invariants, we 
obtain a duality 
\begin{equation} \label{e:catod} (\cgk\mod^{\check{I}})^\vee \simeq 
\widehat{\check{\mathfrak{g}}}_{-\check{\kappa}}\mod^{\check{I}}. 
\end{equation}
Accordingly, an equivalence as in \eqref{e:tfle} at level $\kappa$ 
follows by duality from such an equivalence at level $-\kappa$.   

We now check the compatibility of the above with the assumption of Proposition 
\ref{p:isog} concerning essential images. For dual categories $\scc$ and 
$\scc^\vee$, let us write $\scc^c$ and $(\scc^\vee)^c$ for their full 
subcategories of compact objects, and denote their induced identification by 
\[ \mathbb{D}: \scc^{c, op} \simeq (\scc^\vee)^c. \]

\begin{lem} Fix any $\check{\lambda}$ in $\check{\Lambda}_G$, and write $\rho$ 
	and $\check{\rho}$ for the half sum of the positive roots and coroots of 
	$G$, 
	respectively. With respect to 
	the duality datum \eqref{e:whitd}, we have
	\[ \mathbb{D} j_{\check{\lambda}, !} \simeq j_{\check{\lambda}, *} [-2 
	\langle \check{\rho}, \rho \rangle].\]
	\noindent With respect to the duality datum \eqref{e:catod}, normalized 
	with 
	respect to 
	the Lie algebra of the Iwahori $\check{I}$, we have 
	\[ \mathbb{D} M_{\check{\lambda}} \simeq M_{- \check{\lambda} - 2 
		\check{\rho}}, \]

\end{lem}

The proof of the lemma is straightforward, cf. Lemma 9.8 of \cite{lpw} for the 
assertion regarding Verma modules.

\bibliography{sample}

\newcommand{\etalchar}[1]{$^{#1}$}
\begin{thebibliography}{ABC{\etalchar{+}}18}

\bibitem[AB09a]{arkhipov-bezrukavnikov}
Sergey Arkhipov and Roman Bezrukavnikov.
\newblock {Perverse sheaves on affine flags and {L}anglands dual group}.
\newblock {\em Israel J. Math.}, 170:135--183, 2009.
\newblock With an appendix by Bezrukavrikov and Ivan Mirkovi{\'c}.

\bibitem[AB09b]{ab}
Sergey Arkhipov and Roman Bezrukavnikov.
\newblock Perverse sheaves on affine flags and {L}anglands dual group.
\newblock {\em Israel J. Math.}, 170:135--183, 2009.
\newblock With an appendix by Bezrukavnikov and Ivan Mirkovi\'{c}.

\bibitem[ABC{\etalchar{+}}18]{paris-notes}
Dima Arinkin, Dario Beraldo, Justin Campbell, Lin Chen, Yuchen Fu, Dennis
  Gaitsgory, Quoc Ho, Sergey Lysenko, Sam Raskin, Simon Riche, Nick Rozenblyum,
  James Tao, David Yang, and Yifei Zhao.
\newblock {Notes from the winter school on local geometric Langlands}, 2018.
\newblock Available at
  \url{https://sites.google.com/site/winterlanglands2018/notes-of-talks}.

\bibitem[BDdf]{bdh}
A.~Beilinson and V.~Drinfeld.
\newblock Quantization of {H}itchin's integrable system and {H}ecke
  eigensheaves.
\newblock Available at
  \url{https://math.uchicago.edu/~drinfeld/langlands/hitchin/BD-hitchin.pdf}.

\bibitem[Ber17]{beraldo}
Dario Beraldo.
\newblock Loop group actions on categories and {W}hittaker invariants.
\newblock {\em Advances in Mathematics}, 322:565--636, 2017.

\bibitem[BG08]{fact9}
Alexander Braverman and Dennis Gaitsgory.
\newblock Deformations of local systems and eisenstein series.
\newblock {\em Geometric and Functional Analysis}, 2008.

\bibitem[Bou02]{bour}
Nicolas Bourbaki.
\newblock {\em Lie groups and {L}ie algebras. {C}hapters 4--6}.
\newblock Elements of Mathematics (Berlin). Springer-Verlag, Berlin, 2002.
\newblock Translated from the 1968 French original by Andrew Pressley.

\bibitem[BY13]{by}
Roman Bezrukavnikov and Zhiwei Yun.
\newblock On {K}oszul duality for {K}ac-{M}oody groups.
\newblock {\em Represent. Theory}, 17:1--98, 2013.

\bibitem[Cam18]{campja}
Justin Campbell.
\newblock Compatibility between {J}acquet functors and {L}anglands equivalence.
\newblock \url{http://www.iecl.univ-lorraine.fr/~Sergey.Lysenko
  /notes_talks_winter2018/Ja-2%28Justin%29.pdf}, 2018.

\bibitem[DGK82]{dgk}
Vinay~V. Deodhar, Ofer Gabber, and Victor Kac.
\newblock Structure of some categories of representations of
  infinite-dimensional {L}ie algebras.
\newblock {\em Adv. in Math.}, 45(1):92--116, 1982.

\bibitem[Dhi19]{lpw}
Gurbir Dhillon.
\newblock Semi-infinite cohomology and the linkage principle for
  $\mathscr{W}$-algebras.
\newblock \url{https://arxiv.org/pdf/1905.06477.pdf}, 2019.

\bibitem[Dod11]{dodd}
Christopher~Stephen Dodd.
\newblock {\em Equivariant {C}oherent {S}heaves, {S}oergel {B}imodules, and
  {C}ategorification of {A}ffine {H}ecke {A}lgebras}.
\newblock ProQuest LLC, Ann Arbor, MI, 2011.
\newblock Thesis (Ph.D.)--Massachusetts Institute of Technology.

\bibitem[FF91]{ff-duality}
Boris Feigin and Edward Frenkel.
\newblock Duality in {$W$}-algebras.
\newblock {\em Internat. Math. Res. Notices}, (6):75--82, 1991.

\bibitem[FG]{fg-fusion}
Edward Frenkel and Dennis Gaitsgory.
\newblock {Fusion and Convolution: Applications to Affine Kac-Moody Algebras at
  the Critical Level}.
\newblock {\em Pure and Applied Mathematics Quarterly}, 2(4).

\bibitem[FG06]{fg2}
Edward Frenkel and Dennis Gaitsgory.
\newblock Local geometric {L}anglands correspondence and affine {K}ac-{M}oody
  algebras.
\newblock In {\em Algebraic geometry and number theory}, volume 253 of {\em
  Progr. Math.}, pages 69--260. Birkh\"{a}user Boston, Boston, MA, 2006.

\bibitem[FG09a]{dmod-aff-flag}
Edward Frenkel and Dennis Gaitsgory.
\newblock {$D$}-modules on the affine flag variety and representations of
  affine {K}ac-{M}oody algebras.
\newblock {\em Represent. Theory}, 13:470--608, 2009.

\bibitem[FG09b]{fg}
Edward Frenkel and Dennis Gaitsgory.
\newblock D-modules on the affine flag variety and representations of affine
  {K}ac-{M}oody algebras.
\newblock {\em Representation Theory of the American Mathematical Society},
  13(22):470--608, 2009.

\bibitem[FG09c]{fg-sph}
Edward Frenkel and Dennis Gaitsgory.
\newblock Local geometric {L}anglands correspondence: the spherical case.
\newblock In {\em Algebraic analysis and around}, volume~54 of {\em Adv. Stud.
  Pure Math.}, pages 167--186. Math. Soc. Japan, Tokyo, 2009.

\bibitem[FGV01]{fgv}
Edward Frenkel, Dennis Gaitsgory, and Kari Vilonen.
\newblock Whittaker patterns in the geometry of moduli spaces of bundles on
  curves.
\newblock {\em The Annals of Mathematics}, 153(3):699--748, 2001.

\bibitem[Fie06]{fiebig}
Peter Fiebig.
\newblock {The combinatorics of category O over symmetrizable Kac-Moody
  algebras}.
\newblock {\em Transformation groups}, 11(1):29--49, 2006.

\bibitem[Gai08a]{fact4}
D~Gaitsgory.
\newblock Notes on factorizable sheaves, 2008.

\bibitem[Gai08b]{gtw}
D.~Gaitsgory.
\newblock Twisted {W}hittaker model and factorizable sheaves.
\newblock {\em Selecta Math. (N.S.)}, 13(4):617--659, 2008.

\bibitem[Gai15]{dennis-laumonconf}
Dennis Gaitsgory.
\newblock Outline of the proof of the geometric {L}anglands conjecture for
  {$GL_2$}.
\newblock {\em Ast\'erisque}, (370):1--112, 2015.

\bibitem[Gai16a]{fact3}
Dennis Gaitsgory.
\newblock Eisenstein series and quantum groups.
\newblock In {\em Annales de la Facult{\'e} des sciences de Toulouse:
  Math{\'e}matiques}, volume~25, pages 235--315, 2016.

\bibitem[Gai16b]{quantum-langlands-summary}
Dennis Gaitsgory.
\newblock Quantum langlands correspondence.
\newblock {\em arXiv preprint arXiv:1601.05279}, 2016.

\bibitem[Gai17]{fact6}
Dennis Gaitsgory.
\newblock The semi-infinite intersection cohomology sheaf-ii: the ran space
  version.
\newblock {\em arXiv preprint arXiv:1708.07205}, 2017.

\bibitem[Gai18a]{fact7}
Dennis Gaitsgory.
\newblock A conjectural extension of the kazhdan-lusztig equivalence.
\newblock {\em arXiv preprint arXiv:1810.09054}, 2018.

\bibitem[Gai18b]{fact5}
Dennis Gaitsgory.
\newblock The semi-infinite intersection cohomology sheaf.
\newblock {\em Advances in Mathematics}, 327:789--868, 2018.

\bibitem[Gai18c]{dennispro}
Dennis Gaitsgory.
\newblock Winter {s}chool on {l}ocal {g}eometric {L}anglands theory: program.
\newblock \url{http://www.iecl.univ-lorraine.fr/~Sergey.Lysenko/program_1.pdf},
  2018.

\bibitem[Gai19]{fact1}
Dennis Gaitsgory.
\newblock On factorization algebras arising in the quantum geometric langlands
  theory.
\newblock {\em arXiv preprint arXiv:1909.09775}, 2019.

\bibitem[GGW18]{ggw}
Wee~Teck Gan, Fan Gao, and Martin~H. Weissman.
\newblock L-groups and the {L}anglands program for covering groups: a
  historical introduction.
\newblock Number 398, pages 1--31. 2018.
\newblock L-groups and the Langlands program for covering groups.

\bibitem[GL18a]{gaitsgory-lysenko}
Dennis Gaitsgory and Sergey Lysenko.
\newblock Parameters and duality for the metaplectic geometric langlands
  theory.
\newblock {\em Selecta Mathematica}, 24(1):227--301, 2018.

\bibitem[GL18b]{fact8}
Dennis Gaitsgory and Sergey Lysenko.
\newblock Parameters and duality for the metaplectic geometric langlands
  theory.
\newblock {\em Selecta Mathematica}, 24(1):227--301, 2018.

\bibitem[GL19]{fact2}
D~Gaitsgory and S~Lysenko.
\newblock Metaplectic whittaker category and quantum groups: the" small" fle.
\newblock {\em arXiv preprint arXiv:1903.02279}, 2019.

\bibitem[GPS80]{gps}
Stephen Gelbart and Ilya~I Piatetski-Shapiro.
\newblock Distinguished representations and modular forms of half-integral
  weight.
\newblock {\em Inventiones mathematicae}, 59(2):145--188, 1980.

\bibitem[GR17]{grbook}
Dennis Gaitsgory and Nick Rozenblyum.
\newblock {\em A study in derived algebraic geometry. {V}ol. {I}.
  {C}orrespondences and duality}, volume 221 of {\em Mathematical Surveys and
  Monographs}.
\newblock American Mathematical Society, Providence, RI, 2017.

\bibitem[Hec36]{hecke-announcement}
Erich Hecke.
\newblock {Neure Fortschritte in der Theorie der elliptischen Modulfunktionen}.
\newblock In {\em Comptes rendus du Congr{\`e}s international des
  Mathematiciens Oslo}, pages 140--156. 1936.

\bibitem[Kac90]{kitty}
Victor~G. Kac.
\newblock {\em Infinite-dimensional {L}ie algebras}.
\newblock Cambridge University Press, Cambridge, third edition, 1990.

\bibitem[Lur17]{higheralgebra}
Jacob Lurie.
\newblock {\em Higher Algebra}.
\newblock Available at \url{https://www.math.ias.edu/~lurie/papers/HA.pdf},
  2017.

\bibitem[LY19]{ly}
George Lusztig and Zhiwei Yun.
\newblock Endoscopy for {H}ecke categories and character sheaves.
\newblock \url{https://arxiv.org/pdf/1904.01176.pdf}, 2019.

\bibitem[Lys15]{lysenko-torus}
Sergey Lysenko.
\newblock Twisted geometric langlands correspondence for a torus.
\newblock {\em International Mathematics Research Notices},
  2015(18):8680--8723, 2015.

\bibitem[MS97]{ms}
Dragan Mili\v{c}i\'{c} and Wolfgang Soergel.
\newblock The composition series of modules induced from {W}hittaker modules.
\newblock {\em Comment. Math. Helv.}, 72(4):503--520, 1997.

\bibitem[MV07]{mirkovic-vilonen}
I.~Mirkovi{\'c} and K.~Vilonen.
\newblock Geometric {L}anglands duality and representations of algebraic groups
  over commutative rings.
\newblock {\em Ann. of Math. (2)}, 166(1):95--143, 2007.

\bibitem[Ras15a]{chiralcats}
Sam Raskin.
\newblock {Chiral categories}.
\newblock Available at
  \url{https://web.ma.utexas.edu/users/sraskin/chiralcats.pdf}, 2015.

\bibitem[Ras15b]{cpsi}
Sam Raskin.
\newblock {Chiral principal series categories I: finite dimensional
  calculations}.
\newblock Available at \url{https://web.ma.utexas.edu/users/sraskin/cpsi.pdf},
  2015.

\bibitem[Ras15c]{rdm}
Sam Raskin.
\newblock {$D$}-modules on infinite dimensional varieties.
\newblock {\em Preprint}, \url{web.ma.utexas.edu/users/sraskin/dmod.pdf}, 2015.

\bibitem[Ras16a]{cpsii}
Sam Raskin.
\newblock {Chiral principal series categories II: the factorizable Whittaker
  category}.
\newblock Available at \url{https://web.ma.utexas.edu/users/sraskin/cpsii.pdf},
  2016.

\bibitem[Ras16b]{whit}
Sam Raskin.
\newblock {$\mathscr{W}$}-algebras and {W}hittaker categories.
\newblock \url{https://web.ma.utexas.edu/users/sraskin/whit.pdf}, 2016.

\bibitem[Ras19]{mys}
Sam Raskin.
\newblock {H}omological methods in semi-infinite contexts.
\newblock \url{https://web.ma.utexas.edu/users/sraskin/topalg.pdf}, 2019.

\bibitem[Soe90]{soergel}
Wolfgang Soergel.
\newblock {Kategorie $\mathscr{O}$, perverse Garben und Moduln {\"u}ber den
  Koinvarianten zur Weylgruppe}.
\newblock {\em Journal of the American Mathematical Society}, 3(2):421--445,
  1990.

\bibitem[Zha17]{yifei}
Yifei Zhao.
\newblock {Quantum parameters of the geometric Langlands theory}.
\newblock {\em arXiv preprint arXiv:1708.05108}, 2017.

\end{thebibliography}
\bibliographystyle{alpha}

\end{document}